\documentclass[12pt,reqno]{amsart}
\usepackage{hyperref}
\usepackage{color, bm, amscd, tikz-cd}

\newcommand{\cA}{{\mathcal A}}
\newcommand{\cB}{{\mathcal B}}

\newcommand{\cD}{{\mathcal D}}
\newcommand{\cE}{{\mathcal E}}
\newcommand{\cF}{{\mathcal F}}
\newcommand{\cG}{{\mathcal G}}
\newcommand{\cH}{{\mathcal H}}

\newcommand{\cK}{{\mathcal K}}
\newcommand{\cL}{{\mathcal L}}

\newcommand{\cO}{{\mathcal O}}

\newcommand{\cQ}{{\mathcal Q}}

\newcommand{\cU}{{\mathcal U}}

\newcommand{\bE} {{\mathbb{E}}}
\newcommand{\bC} {{\mathbb{C}}}
\newcommand{\bT}{{\mathbb{T}}}
\newcommand{\sbm}[1]{\left[\begin{smallmatrix} #1
		\end{smallmatrix}\right]}

\newcommand{\bD}{{\mathbb D}}

\newcommand{\lam}{\lambda}
\newcommand{\bA}{{\mathbf A}}
\newcommand{\bB}{{\mathbf B}}
\newcommand{\bfT}{{\mathbf T}}
\newcommand{\bN}{{\mathbf N}}
\newcommand{\bfU}{{\mathbf U}}
\newcommand{\bfV}{{\mathbf V}}

\newcommand{\tA}{\widetilde A}
\newcommand{\tB}{\widetilde B}
\newcommand{\tG}{\widetilde G}
\newcommand{\bcH}{{\boldsymbol{\mathcal H}}}
\newcommand{\bcK}{{\boldsymbol{\mathcal K}}}

\newtheorem{thm}{Theorem}[section]

\newtheorem{corollary}[thm]{Corollary}
\newtheorem{lemma}[thm]{Lemma}

\newtheorem{proposition}[thm]{Proposition}

\theoremstyle{definition}

\newtheorem{definition}[thm]{Definition}
\newtheorem{remark}[thm]{Remark}

\newtheorem{example}[thm]{Example}

\numberwithin{equation}{section}

\def\textmatrix#1&#2\\#3&#4\\{\bigl({#1 \atop #3}\ {#2 \atop #4}\bigr)}
\def\dispmatrix#1&#2\\#3&#4\\{\left({#1 \atop #3}\ {#2 \atop #4}\right)}
\numberwithin{equation}{section}

\def\textmatrix#1&#2\\#3&#4\\{\bigl({#1 \atop #3}\ {#2 \atop #4}\bigr)}
\def\dispmatrix#1&#2\\#3&#4\\{\left({#1 \atop #3}\ {#2 \atop #4}\right)}

\begin{document}
\title[Models for Tetrablock Contractions]{Dilation theory and functional models for tetrablock contractions}
\author[J. A. Ball]{Joseph A. Ball}
\address{Department of Mathematics, Virginia Tech, Blacksburg, VA 24061-0123, USA} \email{joball@math.vt.edu}
\author[H. Sau]{Haripada Sau}
\address{Department of Mathematics, Indian Institute of Science Education and Research Pune, Maharashtra 411008} 
\email{haripadasau215@gmail.com; hsau@iiserpune.ac.in}
\thanks{The research of the second named author was supported by DST-INSPIRE Faculty Fellowship DST/INSPIRE/04/2018/002458.}

\begin{abstract} 
A classical result of Sz.-Nagy asserts that a Hilbert space contraction operator $T$ can be dilated to a unitary $\cU$, i.e.,
 $T^n = P_\cH \cU^n|\cH$ for all $n =0,1,2,\dots$ .  A more general multivariable setting
for these ideas  is the setup where (i) the unit disk is replaced by a domain $\Omega$ contained in ${\mathbb C}^d$, (ii) the contraction
operator  $T$ is replaced by an $\Omega$-contraction, i.e., a commutative operator $d$-tuple $\bfT = (T_1, \dots, T_d)$ on a Hilbert space $\cH$
such that $\| r(T_1, \dots, T_d) \|_{\cL(\cH)} \le \sup_{\lam \in \Omega} | r(\lam) |$ for all rational functions with no singularities in
$\overline{\Omega}$ and the unitary operator $\cU$ is replaced by an $\Omega$-unitary operator tuple, i.e., a commutative operator $d$-tuple
$\bfU = (U_1, \dots, U_d)$ of commuting normal operators with joint spectrum contained in the distinguished boundary $b\Omega$
of $\Omega$.  For a given domain $\Omega \subset {\mathbb C}^d$,  the {\em rational dilation question} asks:  given an $\Omega$-contraction $\bfT$ on $\cH$,
is it always possible to find an $\Omega$-unitary
$\bfU$ on a larger Hilbert space $\cK \supset \cH$ so that, for any $d$-variable rational function without singularities in $\overline{\Omega}$,  one can recover
$r(T)$ as $r(T) = P_\cH r(\bfU)|_\cH$.
We focus here on the case where $\Omega = {\mathbb E}$, a domain in ${\mathbb C}^3$ called the {\em tetrablock}.
(i) We identify a complete set of unitary invariants for a ${\mathbb E}$-contraction $(A,B,T)$ which can then be used to write down a functional model for $(A,B,T)$,
thereby extending earlier results only done for a special case, 
(ii) we identify the class of {\em pseudo-commutative ${\mathbb E}$-isometries} (a priori slightly larger than the class of ${\mathbb E}$-isometries) to which any
${\mathbb E}$-contraction can be lifted, and (iii) we use our functional model to recover an earlier result on the existence and uniqueness of a ${\mathbb E}$-isometric lift
$(V_1, V_2, V_3)$ of a special type for a ${\mathbb E}$-contraction $(A,B,T)$.
\end{abstract}

\dedicatory{\textit{Dedicated to the memory of J\"org Eschmeier: a fine mentor, researcher, and colleague}}

\maketitle

\section{Introduction}
Suppose that we are given  a commutative tuple $\bfT = (T_1, \dots, T_d)$ of operators  on a Hilbert space $\cH$ together with a bounded domain $\Omega$ contained
in $d$-dimensional Euclidean space ${\mathbb C}^d$.   We now recall the notion of $\Omega$ being a {\em spectral set} and $\Omega$ being a {\em complete spectral
set} for the commutative $d$-tuple $\bfT$ and refer to the original paper of Arveson \cite{Arveson} for additional details on this and the related matters which follow.
We say that  $\Omega$ is a {\em spectral set} for $\bfT$ if it is the case that
\begin{equation}   \label{sp-ineq}
    \| r(\bfT) \|_{_{\cB(\cH)}} \le \sup_{\lam \in \Omega} |r(\lam)| \text{ for } r \in \operatorname{Rat}(\Omega)
\end{equation}
where we set $R(\Omega)$ equal to the space of all $d$-variable scalar-valued rational functions $r$ having no singularities in $\overline{\Omega}$
and $\cB(\cH)$ equal to the Banach algebra of all bounded linear operators on $\cH$. Here  $r(\bfT)$ can be defined via the functional calculus given by
$$
r(\bfT) = p(T_1, \dots, T_d) q(T_1, \dots, T_d)^{-1}
$$ 
where $(p,q)$ is a coprime pair of $d$-variable polynomials such that $r = p/q$. 
 In analogy with what happens for the case $\Omega$ equal to
the unit disk ${\mathbb D}$ (see the discussion below), we say simply that $\bfT$ is an $\Omega$-contraction if it is the case that $\Omega$ is a spectral set
for $\bfT$.

We say that $\Omega$ is a {\em complete spectral set} for $\bfT$ if \eqref{sp-ineq} continues to hold when one substitutes matrix rational functions  $R(\lam)
= [ r_{ij}(\lam)]_{i,j=1, \dots, n} = [ p_{ij}(\lam) q_{ij}(\lam)^{-1}]$ having no singularities in $\overline{\Omega}$:
$$
\| R(\bfT) \|_{_{\cB(\cH^n)}} \le \sup_{\lam \in \Omega} \| R(\lam) \|_{_{{\mathbb C}^{n \times n}}}.
$$
The seminal result of Arveson (see \cite{Arveson}) is that $\Omega$ is a complete spectral set for $\bfT$ if and only if there is a commutative $d$-tuple of normal
operators $\bN = (N_1, \dots, N_d)$ on a larger Hilbert space $\widetilde \cK \supset \cH$  with joint spectrum contained in the distinguished boundary
$b\Omega$ of $\Omega$ (in which case we say that $\bN$ is a  {\em $\Omega$}-unitary for short)
so that, for any rational function $r$ with no singularities in $\overline\Omega$ as above, 
it is the case that $r(\bfT)$ on $\cH$ can be represented as the compression of $r(\bN)$ to $\cH$, i.e.,
$$
   r(\bfT) = P_\cH r(\bN) |_\cH
$$
where $P_\cH$ is the orthogonal projection of $\cK$ onto $\cH$.   It is easy to see that a necessary condition for $\Omega$ to be a complete spectral set
for a given operator $d$-tuple $\bfT$ is that $\Omega$ be a spectral set for $\bfT$.  The {\em rational dilation problem} for a given domain $\Omega$ is to determine
if the converse holds:  {\em given $\Omega$, is it always the case that an operator tuple $\bfT$ having $\Omega$ as a spectral set in fact has $\Omega$
as a complete spectral set} (and hence then any $\bfT$ having $\Omega$ as a spectral set has a $d\Omega$-normal dilation $\bN$)?

Let us mention that it is often convenient to reformulate the problem of existence of an {\em $\Omega$-unitary dilation} instead as the problem of
existence of a {\em $\Omega$-isometric lift} (see e.g.~the introduction of \cite{BS-failure}). Here we say that the operator tuple $\bfV = (V_1, \dots, V_d)$ 
on a Hilbert space $\cK$ is a {\em $\Omega$-isometry} if $\bfV$ extends to a $\Omega$-unitary operator tuple $\bfU = (U_1, \dots, U_d)$ 
on a Hilbert spaces $\widetilde \cK \supset \cK$.  We say that $\bfV = (V_1, \dots, V_d)$ on $\cK$  is a {\em lift} of $\bfT = (T_1, \dots, T_d)$ on $\cH$ if
$\cH \subset \cK$ and $r(\bfV)^*|_\cH =r(\bfT)^*$ for $r \in \operatorname{Rat}(\Omega)$, or equivalently, $\bfV$ is a {\em coextension} of $\bfT$  in the sense that
$$
P_\cH r(\bfV)|_\cH = r(\bfT) \text{ and } r(\bfV)\cH^\perp \subset \cH^\perp \text{ for } r \in \operatorname{Rat}(\Omega).
$$
It suffices to consider only {\em minimal} $\Omega$-unitary dilations and {\em minimal} $\Omega$-isometric lifts.  It is always the case that the restriction
of a $\Omega$-unitary dilation to the subspace $\bigvee_{r \in \operatorname{Rat}(\Omega)} r({\mathbf U}) \cH$ gives rise to a minimal
$\Omega$-isometric lift, and conversely, the minimal $\Omega$-unitary extension of a minimal $\Omega$-isometric lift gives rise to a minimal $\Omega$-unitary
dilation for $\bfT$.  Finally we point out that it is often convenient to be more flexible in the definition of an $\Omega$-isometric lift and  of an $\Omega$-unitary dilation
by not insisting that $\cH$ is a subspace of $\cK$ or $\widetilde \cK$ but rather allow an isometric identification map $\Pi \colon \cH \to \cK$
and $\widetilde \Pi \colon \cH \to \widetilde \cK$.  Thus we say that the pair $(\Pi, \bfV)$ is an an $\Omega$-isometric lift for $\bfT$ on $\cH$ if
$\Pi \colon \cH \to \cK$ is an isometric embedding, $\bfV$ is $\Omega$-isometric on $\cK$ and $r(\bfV)^* \Pi = \Pi r(\bfT)^*$ for $r \in \operatorname{Rat}(\Omega)$,
while $(\widetilde \Pi, \bfU)$ is a $\Omega$-unitary dilation of $\bfT$ if $\widetilde \Pi \colon \cH \to \widetilde \cK$ is an isometric embedding,
$\bfU$ is $\Omega$-unitary on $\widetilde \cK$, and $\Pi^* r(\bfU) \Pi = r(\bfT)$ for $ r \in \operatorname{Rat}(\Omega)$.

The motivating classical example for this setup is the case where $\Omega$ is the unit disk ${\mathbb D} \subset {\mathbb C}$.  In this case, the distinguished
boundary $b{\mathbb D}$ of ${\mathbb D}$ is the same as the boundary $\partial {\mathbb D}$ which is the unit circle ${\mathbb T}$
and a $b {\mathbb D}$-normal operator is just a unitary operator.  Since ${\mathbb D}$ is polynomially convex, it suffices to work with polynomials
rather than rational functions with no poles in $\overline{\mathbb D}$. By choosing the polynomial $p$ to be $p = \chi$ and
$\chi(\lam) = \lam$, we see that $\| T \| \le 1$ (i.e., that $T$ be a contraction) is necessary for ${\mathbb D}$ to be a spectral set for $T$.  The fact that
this condition is also sufficient, i.e., that the inequality
$$
   \| p(T) \| \le \sup_{\lam \in {\mathbb D}} |p(\lam)| 
$$
holds for any contraction operator $T$ and polynomial $p$,
is a classical inequality known as von Neumann's inequality going back to  \cite{vonN-Wold}.  to show that ${\mathbb D}$ is a complete spectral set for any
contraction operator $T$, we may use the easier side of Arveson's theorem and show instead that any contraction operator $T$ has a ${\mathbb D}$-unitary dilation. 
But for the case $\Omega = {\mathbb D}$, according to our conventions, a  ${\mathbb D}$-unitary operator is just a unitary operator
$U$ (i.e., $U^*U = U U^* = I_{\widetilde \cK}$).  But any contraction operator $T$ on $\cH$ dilating to a unitary operator $U$ on $\widetilde \cK \supset \cH$
is exactly the content of the Sz.-Nagy dilation theorem (see \cite[Chapter II]{Nagy-Foias}). 

Over the ensuing decades there have been sporadic attempts to find other domains (both contained in ${\mathbb C}$ or more generally contained in ${\mathbb C}^d$)
for which one can settle the rational dilation question one way or the other (i.e., positively or negatively).  Among single-variable domains (as observed in the
introduction of \cite{BS-failure} where precise references are given), it is known that  rational dilation holds if $\Omega \subset {\mathbb C}$ is a simply connected domain 
(simply use a conformal map to reduce to the disk case) or is doubly-connected, but fails if $\Omega$ has two or more holes (see \cite{AHR_Memoir, DM_JAMS}). 
As for multivariable domains,
perhaps the first class to be understood are the polydisks ${\mathbb D}^d$ with $d \ge 2$: for $d = 2$ rational dilation holds due to the And\^o dilation theorem
\cite{ando} while for $d \ge 3$ rational dilation fails (see \cite{Parrott, Varopoulos}). 

More recently the rational dilation problem has been investigated for other concrete multivariable domains originally discovered due to connections
with the $\mu$-synthesis problem in Robust Control Theory (see the original Doyle-Packard paper \cite{DoylePackard} as well as the book \cite{DP} for
a more expository treatment).   We mention in particular the symmetrized bidisk
\begin{equation}  \label{symdisk}
 \Gamma = \{ (s, p) \in {\mathbb C}^2 \colon s = (\lam_1 + \lam_2), \, p  = \lam_1 \lam_2 \text{ for some } (\lam_1, \lam_2) \in {\mathbb D}^2 \}
 \end{equation}
 and a domain in ${\mathbb C}^3$ called the tetrablock and denoted by ${\mathbb E}$:
 \begin{equation}  \label{tetrablock}
\mathbb E:=\left\{(a,b,\text{det}X): X=\begin{bmatrix} a & a' \\ b' & b \end{bmatrix}\text{ with }\lVert X \rVert <1\right\}.
\end{equation}
 As might be expected,  the domain $\Gamma$ behaves like ${\mathbb D}^2$ with respect to the rational dilation problem as both domains are contained in
 ${\mathbb C}^2$: specifically, rational dilation holds for the domain $\Gamma$ (see \cite{AglerYoung00, AglerYoung03, B-P-SR}) and there is a functional model
 analogous to the Sz.-Nagy-Foias model for the disk case (see \cite{AglerYoung03,  BPJOT}), at least for the pure case.
 The situation of the rational dilation problem for the tetrablock ${\mathbb E}$ is less clear:  there is a sufficient and a necessary condition for the existence of a 
 ${\mathbb E}$-isometric lift of a certain form \cite{Tirtha14, BS-failure} but a definitive resolution of the problem in full generality remains elusive 
 (see \cite{BS-failure, PalFailure}). 
However it is shown in \cite{SauNYJM} that, at least in the pure case, it is still possible to construct a functional representation of a pure
$\Gamma$-contraction as the compression to $\cH$ of a certain lift triple $(A_\ell, B_\ell, T_\ell)$ which formally looks like an tetrablock isometry but is not
guaranteed to satisfy all of the required commutativity conditions. 
A similar phenomenon holds for the case where $\Omega = {\mathbb D}^d$ with $d \ge 2$ (see \cite{BallSauDougVol}): for this case, as pointed out above,
there are indeed counterexamples to show that rational dilation fails, but there is nevertheless a weaker type of  lift (called {\em pseudo-commutative
${\mathbb D}^d$-isometric lift)}) which generates a functional model for the given ${\mathbb D}^d$-contractive $d$-tuple $\bfT = (T_1, \dots, T_d)$ even when
rational dilation fails.

In this paper we focus on the case $\Omega = {\mathbb E}$.  As was the case in \cite{Tirtha14}, the most definitive results are for the case of what we shall call
a \textbf{special tetrablock contraction},  i.e., a tetrablock contraction $(A,B,T)$ which has a tetrablock isometric lift $(V_1, V_2, V_3)$ such that $V_3 = V$
is a Sz.-Nagy-Foias minimal isometric lift for the single contraction operator $T$.  As in \cite{Tirtha14}, we identify the additional commutativity conditions
\eqref{com=} which must be imposed on the Fundamental Operator pair $(G_1, G_2)$ of $(A^*,B^*,T^*)$ which characterizes when $(A,B,T)$ is special.
There results a Douglas-type functional model (as in \cite{Doug-Dilation} for the single contraction operator setting) for the tetrablock contraction
which also exhibits the tetrablock isometric lift $(V_1, V_2, V_3)$, all in a functional-model form rather than via block-matrix constructions as in \cite{Tirtha14}.
This Douglas-type model can in turn be converted to a Sz.-Nagy-Foias-type model; the Sz.-Nagy-Foias characteristic function $\Theta_T$ for the
contraction operator $T$, together with the the fundamental operators $(G_1, G_2)$ for the adjoint tetrablock contraction $(A^*, B^*, T^*)$, along
with some additional information needed to handle the case where $T$ is not a pure contraction, form what we call a {\em characteristic tetrablock data set}
for $(A,B,T)$ in terms of which one can write down the functional model.  Conversely, we identify a collection of objects which we call a 
{\em special tetrablock data set}:
specifically,  (i)  a pure contractive operator function $(\cD, \cD_*, \Theta)$, (ii) a pair of operators $(G_1, G_2)$ on the coefficient space $\cD_*$,
(iii) a tetrablock unitary $(R,S,W)$ acting on $\overline{D_\Theta  \cdot L^2(\cD)}$, such that (iv) all these together satisfy a natural invariant-subspace compatibility
condition.  From such a characteristic tetrablock data set we construct a functional model  such that the embedded functional-model operator triple is the most 
general special tetrablock contraction up to unitary equivalence, with its special tetrablock isometric lift also embedded in the functional model.
We also are careful to push the theory as far as we can without the assumption that the original tetrablock contraction is special.  In this case
we identify a class of operator triples $(V_1, V_2, V_3)$ with $V_3$ equal to a minimal isometric lift for $T$ to which $(A,B,T)$ 
can be lifted: here $V_!$ and $V_2$ commute with $V_3$ but not necessarily with each other and it appears that $V_1$, $V_2$ need not be contractions.
In this case there is no converse direction:  there is no guarantee that the compression of a general pseudo-commutative tetrablock isometry $(V_1, V_2, V_3)$
on $\cK$ back to $\cH$ will yield a tetrablock contraction.  

Let us note that the recent paper of Bisai and Pal \cite{BisaiPal21} contains closely related results.  These authors basically compute the $Z$-transform of 
 the Sch\"affer-type construction of the unique special tetrablock isometric lift $(V_1, V_2, V_3)$ (where $V_3$ is equal to the minimal Sz.-Nagy isometric lift
 of $T$) to arrive at a functional model for this lift.  Our approach on the other hand
 uses the Douglas lifting approach to construct the functional model directly with the existence and uniqueness of the special tetrablock isometric lift 
 falling out as part of the construction.  When the tetrablock contraction is not special and no such lift is possible, the same construction still leads
 to a functional model but $(V_1, V_2, V_3)$ is only a pseudo-commutative tetrablock isometry and there is no tetrablock isometric lift constructed in this way.
 The results for the special case arise as a special case (the case where the Fundamental Operator pair $(G_1, G_2)$ for the tetrablock contraction $(A^*,B^*,T^*)$ 
 satisfy the additional commutativity conditions \eqref{com=}) of the general functional-model construction.  The paper \cite{BisaiPal21} also obtains a 
 noncommtative functional model for a non-special case, based on the work of Durszt \cite{Durszt} (a variation of the approach of Douglas for the construction
 of the minimal isometric lift for the case of a single contraction operator $T$), but with the additional hypothesis that $A$ and $B$ commute not only with $T$
 but also with $T^*$.  It is clear that the complete unitary invariant for a pure tetrablock contraction $(A,B,T)$ consists of the characteristic function $\Theta_T$ of $T$
 together with the Fundamental Operator pair $(G_1, G_2)$ of $(A^*, B^*, T^*)$; for the non-pure case (where $\Theta_T$ is no longer inner) we add
 a certain tetrablock unitary $(R,S,W)$ acting on $\overline{\Delta_T H^2(\cD_T) }$ which is part of our model (see Theorem \ref{Thm:CompUniInv} below),
 while Bisai-Pal add the  Fundamental Operator pair $(F_1, F_2)$ for $(A,B,T)$ and argue that $(\Theta_T, (F_1, F_2), G_1, G_2))$ is a complete unitary
 invariant.   It remains to be seen which is the more relevant and useful in the future.

It is now becoming clear that the domains ${\mathbb D}^d$ (polydisk), $\Gamma$ (symmetrized bidisk), ${\mathbb E}$ (tetrablock) as well as ${\mathbb D}^d_s$
(symmetrized polydisk) all have common features with respect to the associated operator theory and applications to the rational dilation problem 
for each of these domains. The paper \cite{BallSauDougVol} shows how a program completely parallel to that done here for the tetrablock case
can be worked out equally well  for the polydisk case $\Omega = {\mathbb D}^d$ (where rational dilation is known to fail when $d \ge 3$). 
In all these settings, there appear a  pair of unitary invariants called Fundamental Operators which play a key role  as part of a set (including the Sz.-Nagy--Foias
characteristic function of an appropriate contraction operator determined by the operator tuple)  of unitary invariants for the operator tuple of whatever class.
The notion of Fundamental Operators as a fundamental object of interest seems to have appeared first in connection with the symmetrized bidisk
$\Gamma$ \cite{BPJOT}, then in connection with the un-symmetrized polydisk \cite{sauAndo,  BallSauDougVol}, and now also in connection with the
symmetrized polydisk (see \cite{Pal21}).
Often the proper notion of Fundamental Operators for one setting is found by making a 
correspondence of the less understood setting with some other better understood setting,  and then adapting definitions for the first to become
definitions for the second.  In particular, many of the results for the tetrablock case were originally found by adapting from results for the
symmetrized bidisk case (see e.g.~\cite{Tirtha14}), and it has been shown how one can deduce the bidisk functional model from the tetrablock
functional model (see \cite{sauAndo}).  In this spirit in a future publication we plan to show  how the results
from \cite{BallSauDougVol} for the polydisk case (most of which are just statements parallel to what is done here for the tetrablock case)
 can alternatively be derived as a corollary of the corresponding results for the tetrablock case via the simple observation:
if $\bfT = (T_1, \dots, T_d)$ is a commutative, contractive operator $d$-tuple, then for $1 \le i \le d$, if we set $T_{(i)} = \Pi_{1 \le j \le d \colon j \ne i} T_j$, then for each
$i = 1, \dots, d$ the $d$-tuple $(A_i, B_i, P) = (T_i, T_{(i)}, \Pi_{1 \le j \le d} T_j)$ is a tetrablock contraction; the $d=2$ case can be found in \cite[Section 3, Version 3]{sauAndo}.

Finally,  let us point out that
it is possible to reformulate the rational dilation problem for a given domain $\Omega$ as a problem about unital representations of a unital function algebra:  
given a contractive representation
$\pi \colon f \in \cA \mapsto \pi(f) \in \cB(\cH)$ which is contractive ($\| \pi(f) \|_{\cB(\cH)}  \le \| f \|_{\cB(\cH)}$ where the unital representation property is that
$\pi(1_\cA) = I_\cH$ and $\pi(f_1 \cdot f_2) = \pi(f_1) \pi(f_2)$, is it automatically the case that the representation is {\em completely contractive}, i.e.,
still contractive after tensoring with ${\mathbb C}^{n \times n}$ for any $n \in {\mathbb N}$?  To recover the original formulation as a special case,
one can take $\cA = \overline{\operatorname{Rat}(\Omega)}$ where the closure is in the $C^*$-algebra $C(b\Omega)$
(continuous functions on the distinguished boundary $b \Omega$).  However with this more general formulation one can consider function algebras
which go beyond $\overline{\operatorname{Rat}({\Omega})}$, e.g.,  the
constrained subalgebra ${\mathbb C} \cdot 1 + z^2 {\mathbb A}({\mathbb D})$ of the disk algebra ${\mathbb A}({\mathbb D}) =
\overline{\operatorname{Rat}({\mathbb D})}$.  Alternatively, it is often possible to represent the algebra $\cA$ as conformally equivalent to the algebra of all 
functions analytic on some algebraic curve $\Omega = {\mathbf C}$ embedded in some higher-dimensional closed complex manifold (the Neil
parabola  intersected with the bidisk for the case of ${\mathbb C} \cdot 1 + z^2 {\mathbb A}({\mathbb D})$).  
For the state of knowledge (up to 2018) on this direction of dilation theory
including much discussion and references on earlier work, we refer to the paper of Dritschel and Undrakh \cite{DU}.  We shall not pursue this direction here.

The paper is organized as follows.  After the present Introduction, in Section 2  we collect assorted definitions and illustrative results concerning
tetrablock contractions, tetrablock isometries, and tetrablock unitaries, including a direct proof of the existence of the Fundamental Operator pair
for a given tetrablock contraction, which will be needed in the sequel.  Here we also show how to associate a tetrablock unitary $(R,S,W)$ with a
tetrablock contraction $(A,B,T)$ in a canonical way; this is the key ingredient needed to eliminate the {\em purity} assumption on the contraction operator
$T$ required in earlier work on this problem (see \cite{SauNYJM}).
Section 3 shows how a lifting framework for the tetrablock-contraction setting can be constructed as an embellishment of the Douglas-model lifting framework
\cite{Doug-Dilation} originally formulated as an approach to the Sz.-Nagy dilation theorem for a single contraction operator $T$, with the
pseudo-commutative tetrablock-isometric lift $(V_1, V_2, V_3)$ having $V_3 = T$ and $V_1$ and $V_2$ constructed by making use of the
Fundamental Operator pair for the tetrablock contraction $(A^*, B^*, T^*)$.  The final Section 4 identifies the invariants required to write down a functional
model equipped with a model operator triple $(A,B,T)$ which is concrete functional-model version of a general tetrablock contraction.

\section{The fundamentals of tetrablock contractions}
This section gives a brief introduction to the operator theory associated with the tetrablock.

\subsection{Tetrablock contractions}
The {\bf tetrablock}, denoted by ${\mathbb E}$, is the  non-convex but polynomially convex domain in $\mathbb C^3$ given by \eqref{tetrablock}.  From this formula 
for ${\mathbb E}$ is is easy to read off the following symmetry properties.

\begin{proposition}  \label{P:symmetries}
The tetrablock ${\mathbb E}$ has the following symmetry properties:
\begin{enumerate}
\item ${\mathbb E}$ is invariant under complex conjugation:
$$
(a,b,t) \in {\mathbb E} \Leftrightarrow  (\overline{a}, \overline{b}, \overline{t}) \in {\mathbb E}.
$$
\item ${\mathbb E}$ is invariant under interchange of the first two coordinates:
$$
(a,b,t) \in {\mathbb E}  \Leftrightarrow (b,a,t) \in {\mathbb E}.
$$
\end{enumerate}
\end{proposition}

The distinguished boundary of ${\mathbb E}$, i.e., the \v Silov boundary with respect to the algebra of functions that are analytic in $\bE$ and 
continuous on $\overline{\bE}$, is given by
$$
b\bE:=\left\{(a,b,\text{det}X): X=\begin{bmatrix} a & a' \\ b' & b \end{bmatrix}\text{ is a unitary }\right\}
$$
(see \cite[Theorem 7.1]{awy07}).  From this characterization it is easy to see that $b{\mathbb E}$ is also invariant under the two involutions
$(a,b,t) \mapsto (\overline{a}, \overline{b}, \overline{t})$ and $(a,b,t) \mapsto (b,a, t)$.

 Several tractable characterizations of the tetrablock can be found in \cite[Theorem 2.2]{awy07}; 
we pick two of these that will be used in what follows.

\begin{thm}\label{T:CharacTetra} For a point $(a,b,t)\in\bC^3$, the following are equivalent:
\begin{itemize}
    \item[(i)] $(a,b,t)\in\bE;$
    \item[(ii)] with the rational function $\Psi:\overline{\bD}\times\bC^3\to\bC$ defined as
    \begin{align}\label{Psi}
        \Psi(z,(a,b,t))=\frac{a-z t}{1-z b},
    \end{align}$\sup_{z\in\overline{\bD}}|\Psi(z,(a,b,t))|<1$; and if $ab=t$ then, in addition, $|b|<1$;
    \item[(iii)] with $\Psi$ as in \eqref{Psi}, $\sup_{z\in\overline{\bD}}|\Psi(z,(b,a,t))|<1$; and if $ab=t$ then, in addition, $|a|<1$.
\end{itemize}Moreover, when item (i) is replaced by $(a,b,t)\in\overline{\bE}$, then all the strict inequalities in items (ii) and (iii) are replaced by non-strict inequalities. 
\end{thm}

\begin{remark}  \label{R:symmetry} Note that in Theorem \ref{T:CharacTetra},  the equivalence of (i) $\Leftrightarrow$ (iii) is an immediate consequence of the equivalence of 
(i) $\Leftrightarrow$ (ii) in view of the
invariance of ${\mathbb E}$ under the involution $(a,b,t) \mapsto (b,a,t)$.
\end{remark}

Recall the notions of \textbf{tetrablock unitary}, \textbf{tetrablock
isometry} and \textbf{tetrablock contraction} given in the Introduction.  
Several algebraic characterizations of tetrablock isometries and tetrablock unitaries are known; see Theorems 5.4 and 5.7 in \cite{Tirtha14}. 
We recall the ones that are useful for our purposes here.  Here we use the notation 
$r(X)$ for the \textbf{spectral radius} of a Hilbert-space operator $X$.  It is then not difficult to see that the ${\mathbb E}$-symmetries noted in
Proposition \ref{P:symmetries} imply the same symmetries on the respective operator classes (with respect to the class of ${\mathbb E}$-isometries
which requires a little extra care), as noted in the next result. We leave the
easy verification as an exercise for the reader.

\begin{proposition} \label{P:operatorsym}
Suppose that $(A,B,T)$ is a triple of bounded operators on a Hilbert space $\cH$.
Then:
\begin{enumerate}
\item $(A,B,T)$ is a ${\mathbb E}$-contraction $\Leftrightarrow$ $(A^*, B^*, T^*)$ is a ${\mathbb E}$-contraction $\Leftrightarrow$
$(B,A,T)$ is a ${\mathbb E}$-contraction.

\item $(A,B,T)$ is a ${\mathbb E}$-isometry $\Leftrightarrow$ $(B,A,T)$ is a ${\mathbb E}$-isometry.

\item $(A,B,T)$ is a ${\mathbb E}$-unitary $\Leftrightarrow$ $(A^*, B^*, T^*)$ is a ${\mathbb E}$-unitary $\Leftrightarrow$ $(B,A,T)$ is a ${\mathbb E}$-unitary.
\end{enumerate}
\end{proposition}

\begin{thm}\label{T:IsoChar}
Let $(A,B,T)$ be a commutative triple of bounded Hilbert space operators. Then the following are equivalent:
\begin{itemize}
    \item[(i)] $(A,B,T)$ is a tetrablock isometry (respectively unitary);
    \item[(ii)] $(A,B,T)$ is a tetrablock contraction and $T$ is an isometry (respectively unitary); 
    \item[(iii)] $A=B^*T$, $B$ is a contraction and $T$ is an isometry (respectively unitary); and
    \item [(iv)] $B=A^*T$, $A$ is a contraction and $T$ is an isometry (respectively unitary).
    \item[(v)] $B = A^* T$, $r(A) \le 1$ and $r(B) \le 1$, and $T$ is an isometry (respectively unitary).
\end{itemize}
\end{thm}

\subsection{Pseudo-commutative tetrablock isometries and unitaries}   \label{SS:pc-tetra-isom/unit}
We propose to introduce the notions of \textbf{pseudo-commutative tetrablock unitary}  and \textbf{pseudo-commutative tetrablock isometry} 
for an operator triple $(A,B,T)$ by using criterion (iii) or equivalently (iv) in Theorem \ref{T:IsoChar} but with the weakening the commutativity hypothesis imposed on the
whole triple $(A,B,T)$ to just the condition that $A$ and $B$ commute with $T$ (but not necessarily with each other).  As we are also dropping the condition
that $A$ or $B$ be a contraction, a more proper term would be {\em noncontractive pseudo-commutative tetrablock isometry}, but, as this term will be
consistent throughout, we settle on the shorter term for brevity.  The resulting definition is as
follows.  We leave it to the reader to verify that the two formulations are equivalent.

\begin{definition} \label{D:pc}
 Let  $(A,B,T)$ be a triple of bounded Hilbert-space operators. We say that the triple
$(A,B,T)$ is a \textbf{pseudo-commutative tetrablock isometry} (respectively, \textbf{unitary}) if any of the following equivalent conditions holds:

\begin{enumerate}
\item  $T$ is an isometry (respectively, unitary) and
\begin{equation}  \label{PseudoIdentity1}
    AT=TA,\quad BT=TB, \quad  A = B^*T.
    \end{equation}

\item $T$ is an isometry (respectively, unitary), and 
\begin{equation} \label{PseudoIdentity2}
AT = TA, \quad BT = TB, \quad B = A^*T.
\end{equation}
\end{enumerate}
\end{definition}

\begin{remark} \label{R:pctetra-vs-tetra}
From Definition \ref{D:pc} and Theorem \ref{T:IsoChar},  we see that any tetrablock isometry/unitary is also a pseudo-commutative tetrablock isometry/unitary
but not conversely.   If we wish to emphasize that we are referring to the logically more special {\em tetrablock isometry/ unitary} rather than the more general
{\em pseudo-commutative tetrablock isometry/unitary},  we often will  say
{\em strict tetrablock isometry/unitary} for emphasis.
\end{remark}

\begin{remark}  \label{R:pseudocom} We now present a couple of elementary observations on pseudo-commutative versus strict tetrablock unitaries
unitaries  which we hope give the reader some additional insight.  

\smallskip
\noindent
(1)  We remark that {\em if $(A,B,T)$ is a pseudo-commutative tetrablock unitary, then $A$ and $B$ are not necessarily normal operators} as would happen
in the strict case. For example, pick a non-normal 
contraction $G_1$ acting on a Hilbert space $\cE$ and consider the triple $(M_{G_1^*},M_{\zeta G_1},M_\zeta)$ on $L^2(\cE)$. It is easy to see that this triple is a 
pseudo-commutative tetrablock unitary. However, neither $A$ nor $B$ is normal unless $G_1$ is so. Also note that if $(A,B,T)$ is a pseudo-commutative tetrablock 
unitary, then so is the adjoint triple $(A^*,B^*,T^*)$. 
This can be seen by observing that the adjoint of the identities in \eqref{PseudoIdentity1} with $T$ unitary can be converted to the identities
\eqref{PseudoIdentity2} for   $(A^*,B^*,T^*)$ with $T^*$ still unitary.    Note next that if $(A,B,T)$ is a pseudo-commutative tetrablock unitary, then
\begin{align*}
&A^*A = T^* B B^* T= B T^* T B^* = B B^*, \\
&  B^*B = T^* A A^* T = A T^* T A^* = A A^*.
\end{align*}
Thus we always  have
\begin{equation}  \label{pc1}
A^*A = BB^*, \quad AA^* = B^* B
\end{equation}
for a pseudo-commutative tetrablock unitary $(A,B,T)$.  
As a first consequence of \eqref{pc1}, we see that
if $A$ is normal, then
$$
B^*B = A A^* = A^* A = B B^*
$$
and $B$ is also normal.  Similarly if $B$ is normal, then $A$ is also normal.  In conclusion, 
{\em if $(A,B,T)$ is a pseudo-commutative unitary such that one of $A$ or $B$ is normal, then so is the other.}

\smallskip

\noindent
(2) 
We note as a consequence of (iii) $\Rightarrow$ (i) in Theorem \ref{T:IsoChar} that in particular {\em if $(A, B, T)$ is a
strict tetrablock unitary  
(so we also have $AB = BA$), then the operators $A$ and $B$ are normal. }  One can see this directly from the considerations here as follows.
As a strict tetrablock unitary in particular meets  all the requirements for membership in the {\em pseudo-commutative tetrablock unitary} class,
we know that \eqref{pc1} holds.
Combining this with the commutativity relation $AB = BA$  then gives us
$$
 A^* A = A^* B^* T = B^* A^* T   = B^* B = A A^*
$$
showing that $A$ is normal. The same computation with the roles of $A$ and $B$ interchanged then shows that $B$ is also normal.
The full strength of (iii) $\Rightarrow$ (i) in Theorem \ref{T:IsoChar} is that in addition the commutative normal triple $(A,B,T)$ has joint spectrum
in the boundary of the tetrablock ${\mathbb E}$; for this somewhat deeper fact  we refer to \cite{Tirtha14}.
\end{remark}

The next result gives a feel for how close pseudo-commutative tetrablock isometries come to being strict tetrablock isometries.

\begin{thm} \label{T:pc-vs-strict-Eisom} Let $(A,B,T)$ be a pseudo-commutative tetrablock isometry on a Hilbert space $\cH$.

\begin{enumerate}
\item Then the spectral radius $r(AB)$  of the product operator $AB$ is given by
\begin{equation}  \label{norm/rad}
   r(AB) = \operatorname{max} \{ \| A \|^2, \, \| B \|^2 \|.
\end{equation}

\item Suppose in addition that $AB = BA$ and $r(A) \le 1$, $r(B) \le 1$.  Then both $A$ and $B$ are contraction operators
($\max \{ \| A \|, \| B \| \} \le 1$) and $(A,B, T)$ is a strict tetrablock isometry.
\end{enumerate}
\end{thm}

\begin{proof}  The proof follows the ideas of Bhattacharyya \cite[pp.1619-1620]{Tirtha14}.
We first consider statement (1).  Form two operators $X_1 = \sbm{ 0 & A  \\ B & 0 }$ and $X_2 = \sbm{ T & 0 \\ 0 & T}$ on $\sbm{ \cH \\ \cH}$.
From the two relations $B^*T = A$ and $A^*T = B$ we deduce that $X_1 = X_1^* X_2$ where $X_2 X_2^* = \sbm{ T T^* & 0 \\ 0 & T T^*} \preceq 
\sbm{ I & 0 \\ 0 & I}$ since $T$ is an isometry.  Hence 
$$
 X_1 X_1^*  = X_1^* X_2 X_2^* X_1  \preceq X_1^* X_1,
 $$
 i.e.,  $X_1$ is a {\em hyponormal operator}.  By a theorem of Stampfli (see \cite[Proposition 4.6]{Conway}, it follows that $r(X_1) = \| X_1 \|$. 
 We compute the operator norm of $X_1$ as follows:
 $$
  \| X_1\|^2 = \bigg\| \begin{bmatrix} 0 & A \\ B & 0 \end{bmatrix} \begin{bmatrix} 0 & B^* \\ A^* & 0 \end{bmatrix} \bigg\|
  = \bigg\| \begin{bmatrix} A A^* & 0 \\ 0 & B B^* \end{bmatrix} \bigg\|  = \operatorname{max} \{ \| A \|^2, \| B \\^2 \}.
  $$
  and hence $\| X_1 \| = \max \{ \| A \|, \| B \| \}$.  To compute $r(X_1)$, note first that $X_1^*X_1  = \sbm{ A B & 0 \\ 0 & BA }$ and hence
  $$
   X_1^{2n} = \begin{bmatrix} (AB)^n & 0  \\ 0 & (BA)^n \end{bmatrix}.
   $$
   Consequently, 
   $$
   r(X_1) = \lim_{n \to \infty} \max \|  \| (AB)^n \|^{\frac{1}{2n}}, \| (BA)^n \|^{\frac{1}{2n}} \} = \max \{ r(AB)^{\frac{1}{2}}, r(BA)^{\frac{1}{2}} \}.
   $$
However a general fact is that the nonzero spectrum of $AB$ is the same as the nonzero spectrum of $BA$, and hence $r(AB) = r(BA)$.
Thus  $r(X_1) = \| X_1\|$ gives us  \eqref{norm/rad} and the proof of statement (1) is complete.

\smallskip

As for statement (2), a known fact is that if $A$ and $B$ commute, then the spectrum of the product operator $AB$ is given by
$$
  \sigma(AB) = \{ \lam \cdot \mu \colon \lam \in \sigma(A), \, \mu \in \sigma(B) \}.
$$
 Hence the hypothesis that $r(A) \le 1$ and $r(B) \le 1$ implies that $r(AB) \le 1$ as well.  But then from the conclusion of statement (1)
 already proved, we conclude that both $A$ and $B$ are contraction operators, and the proof of statement (2) is now complete.
(Note that this also proves (v) $\Rightarrow$ (iii) or (iv) in Theorem \ref{T:IsoChar}.)
\end{proof}

\begin{example}  \label{E:model}

(1) \textbf{A pseudo-commutative/strict tetrablock isometry.}
Let $\cE$ be a coefficient Hilbert space and $H^2(\cE) = H^2 \otimes \cE$ be the associated Hardy space of $\cE$-valued functions.
Let $G_1$ and $G_2$ be operators on $\cE$ and set 
\begin{equation}   \label{pcform}
 A  = M_{G_1^*  + z G_2}, \quad B = M_{G_2^* + z G_1}, \quad T = M_z \text{ on } H^2(\cE).
\end{equation}
Then it is immediate that $T$ is an isometry and that $A$ and $B$ commute with $T$.  The special coupled form of the pencils defining $A$ and $B$ 
enables us to show that $A = B^*T$: 
\begin{align*}
B^* T&  = (I_{H^2} \otimes G_2^* + M_z \otimes G_1)^* \cdot  (M_z \otimes I_\cG) \\
&  = (I_{H^2} \otimes G_2 + M_z^* \otimes G_1^*)  \cdot (M_z \otimes I_\cG) \\
 & =  (I_{H^2} \otimes G_1^*) + (M_z \otimes G_2)
  = M_{G_1^* + z G_2}  = A
\end{align*}
and similarly $B = A^* T$.  

For $(A,B,T)$ to be a strict tetrablock isometry, we need in addition that $AB = BA$ and that $\|A \| \le 1$ (in which case also
$\| B \| = \| A^* T \| \le 1$ as well).  To ensure that $\| A \| \le 1$ requires that $G_1$ and $G_2$ are not too large in the precise sense that
\begin{equation}  \label{PencilContr}
\sup_{z \in {\mathbb T}} \| G_1^* + z G_2 \| \le 1.
\end{equation}
To check the condition $AB = BA$, we compute
\begin{align*}
A B & = M_{G_1^*  + z G_2} M_{G_2^* + z G_1}  \\
&  = (I_{H^2} \otimes G_1^* + M_z \otimes G_2) \cdot (I_{H^2} \otimes G_2^* + M_z \otimes G_1)  \\
& = I_{H^2} \otimes G_1^* G_2^* + M_z  \otimes (G_1^* G_1 + G_2 G_2^*) + M_z^2 \otimes G_1 G_2
\end{align*}
while a similar computation gives us
$$
BA = I_{H^2} \otimes G_2^* G_1^* + M_z \otimes (G_1 G_1^* + G_2^* G_2) + M_z^2 \otimes G_2 G_1.
$$
We conclude that in this example,  $(A, B, T)$ is a strict tetrablock isometry exactly when \eqref{PencilContr} together with the following commutativity conditions hold:
\begin{equation}  \label{com=}
G_1 G_2 = G_2 G_1, \quad  G_1^* G_1 + G_2 G_2^* = G_1 G_1^* + G_2^* G_2,
\end{equation}
sometimes also written more compactly in terms of commutators as
$$
 [ G_1, G_2] = 0, \quad [G_1^*, G_1] = [G_2^*, G_2]
$$
where in general $[X,Y]$ is the commutator:
$$
  [ X, Y] = X Y - Y X.
$$
\smallskip

(2) \textbf{A pseudo-commutative/strict tetrablock unitary.}
It is easy to  use the spectral theory for unitary operators (a particular case of the spectral theory of normal operators) to write down a model for the general
pseudo-commutative/strict tetrablock unitary $(R,S,W)$.  as follows.  By Definition \ref{D:pc} we see in particular that $W$ is unitary.
By the spectral theory for general normal operators (see e,g,~any of \cite{Dix, Arv} or \cite[Chapter 2]{Arveson02}),  after a unitary change of
coordinates, we can represent  $W$ as the operator $M_\zeta$ of multiplication by the coordinate function ($M_\zeta \colon h(\zeta) \mapsto \zeta h(\zeta)$) on a 
direct-integral space $\bigoplus\int_{\mathbb T} \cH_\zeta\, \nu({\tt d}\zeta)$ determined by a scalar spectral measure $\nu$ supported on ${\mathbb T}$ and a 
measurable multiplicity function
$\zeta \mapsto \dim \cH_\zeta$..  Since the operators $R$ and $S$ commute with $W = M_\zeta$, it follows that $R$ and $S$
are represented as {\em decomposable operators} on $\bigoplus \int_{\mathbb T} \cH_\zeta\, \tt{d} \nu(\zeta)$,  i.e.,  
$R = M_\phi \colon h(\zeta) \mapsto \phi(\zeta) h(\zeta)$ and $S = M_\psi \colon h(\zeta) \mapsto \psi(\zeta) h(\zeta)$
for measurable functions such that $\phi(\zeta) \in \cB(\cH_\zeta)$, $\psi(\zeta) \in \cB(\cH_\zeta)$ for a.e.~$\zeta$.  
The fact that in addition $R = S^*W$ then forces $\phi(\zeta) = \psi(\zeta)^*  \cdot \zeta $ for a.e.~$\zeta$,   Thus any pseudo-commutative tetrablock unitary has the form
\begin{equation}  \label{gen-tetra-unitary}
 (R, S, W) = (M_{\psi^* \cdot \zeta}, M_\psi, M_\zeta) \text{ acting on } \bigoplus \int_{\mathbb T} \cH_\zeta \, \nu(\tt{d}  \zeta)
\end{equation}
If $(R,S,W)$ is a strict tetrablock unitary, then in addition we must have that $\psi(\zeta)$ is a {\em contractive normal operator} on $\cH_\zeta$ for a.e.~$\zeta$ 
in order to guarantee in addition that $RS = SR$ and that $\| R \| \le 1$, $\| S \| \le 1$.  By this analysis we conclude that \eqref{gen-tetra-unitary} 
(with $\psi(\zeta)$ constrained to be contractive normal  for a.e.~$\zeta$ for the strict case) 
is the general form for a pseudo-commutative/strict tetrablock unitary.  In a less-functional form, to write a pseudo-contractive tetrablock triple $(R,S,W)$, the free parameters
are:  (i) a unitary operator $W$, and (ii) an operator $S$ commuting with $W$; then the associated pseudo-commutative tetrablock contraction is
$(W^*S, S, W)$; for this to be strict, one must require in addition that the operator $S$ in the commutant of $W$  be a normal contraction.
\end{example}

To deduce the von Neumann-Wold decomposition for a tetrablock isometry, the next lemma is useful. Only the special case where the operator $S$
in the statement is a shift will be needed for our application, in which case the result is well-known (see e.g. \cite[page 22]{NFintertwine}).
For completeness we present here a proof of the general result.

\begin{lemma} \label{L:zero}
Let $W$ be a unitary operator on $\cH_2$, $S$ an operator on $\cH_1$ such that $S^{*n}\to 0$ in the strong operator topology as $n\to\infty$. 
If $X $ is a bounded operator from $\cH_2$ to $\cH_1$ such that $XW=SX$, then $X=0$.
\end{lemma}

\begin{proof} From $XW = SX$ we get by iteration that $XW^n = S^n X$ for $n = 1,2,\dots$. Taking adjoints gives then $W^{*n} X^* = X^* S^{*n}$.
Apply this identity to an arbitrary fixed vector $x \in \cH_2$ to get
$ W^{*n} X^* x = X^* S^{*n} x$ for all $n \ge 1$.  Apply $W^n$ to both sides of this equation to get $W^n W^{*n} X^*x = W^n X^* S^{*n} x$ for $n=1,2, \dots$.
As $W$ is unitary, this becomes $X^*x  = W^n X^* S^{*n} x$.  Taking norms then gives
$$
 \| X^* x \| = \| W^n X^* S^{*n} x \| = \| X^* S^{*n} x \| \le \| X^*\| \| S^{*n} x \| \to 0 \text{ as } n \to \infty
 $$
 by the assumed strong convergence of powers of $S^*$ to zero, forcing $X^*$ (and hence also $X$) to be the zero operator.
 \end{proof}

The von Neumann-Wold decomposition (see   \cite{Wold, vonN-Wold, Nagy-Foias}) ensures that if $T$ is an isometry acting on a Hilbert space $\cH$, then 
$T$ can be represented as an operator as the external direct sum $M_z \oplus U$ of a shift operator $M_z$ acting on a Hardy space $H^2(\cE)$
and a unitary operator $W$ on $\cF$ for some coefficient Hilbert spaces $\cE$ and $\cF$.
 The following result not only gives a model for an arbitrary pseudo-commutative/strict tetrablock isometry $(A,B,T)$, but also can be seen 
 as a pseudo-commutative/strict tetrablock-isometry analogue of the classical von Neumann--Wold decomposition for a
 single  isometric Hilbert space operator $T$.
 
\begin{thm}\label{T:WoldPC}  Let $(A,B,T)$ be an operator-triple on the Hilbert space $\cH$.

\smallskip

\noindent
{\rm(1)}
Then $(A,B,T)$ is a pseudo-commutative tetrablock isometry on $\cH$ if and only if there exist Hilbert spaces $\cE$, $\cF$, operators $G_1,G_2$ acting on $\cE$
subject to \eqref{PencilContr}, along with a
pseudo-commutative tetrablock unitary $(R,S,W)$ acting on $\cF$, such that $\cH$ is isomorphic to $\sbm{H^2(\cE) \\ \cF}$ and under the same isomorphism
$(A,B,T)$ is unitarily equivalent to
\begin{align}\label{Pencils}
\left(\begin{bmatrix}M_{G_1^*+z G_2}&0\\0&R \end{bmatrix},\begin{bmatrix}M_{G_2^*+z G_1}&0\\0&S \end{bmatrix},\begin{bmatrix}M_z &0\\0&W \end{bmatrix}\right).
\end{align}

\smallskip

\noindent
{\rm (2)}  Then $(A,B,T)$ is a strict tetrablock isometry on $\cH$ if and only if $(A,B,T)$ is unitarily equivalent to the operator triple as in \eqref{Pencils}
(with $G_1$, $G_2$ subject to \eqref{PencilContr})
acting on a space $\sbm{H^2(\cE) \\ \cF}$, where in addition the operator-pencil coefficients $(G_1, G_2)$ satisfy the system of operator identities
\eqref{com=}, and the triple $(R,S, W)$ is a strict tetrablock unitary (i.e., we also have the relation $RS = SR$ with $R$ and $S$ contraction operators).
\end{thm}

\begin{remark} \label{R:FuncModel} We shall think of a triple of operators on $\sbm{ H^2(\cE) \\ \cF}$ as in \eqref{Pencils} as a \textbf{functional model} for a pseudo-commutative/strict
tetrablock isometry/unitary.  The $H^2(\cE)$-component clearly has a functional form while the second component can be brought to a measure-theoretic
functional form as in item (2) in Example \ref{E:model}.
\end{remark}

\begin{proof}
The sufficiency (for both the pseudo-commutative and the strict case)  follows from Example \ref{E:model}.

We now suppose that $(A,B,T)$ is a strict tetrablock isometry.
Let us  apply the Wold decomposition to the isometry $T$: there exist Hilbert spaces 
$\cE$, $\cF$, and a unitary $\tau:\cH\to\sbm{H^2(\cE)\\ \cF}$ such that
$$
\tau T\tau^*=
\begin{bmatrix}M_z&0\\0&W \end{bmatrix}:\begin{bmatrix} H^2(\cE)\\ \cF\end{bmatrix} \to \begin{bmatrix} H^2(\cE)\\ \cF\end{bmatrix}
$$
for some unitary $W$ on $\cF$. Next assume that
$$
\tau(A,B)\tau^*=\left(\begin{bmatrix}
A_{11}&A_{12}\\A_{21}&R
\end{bmatrix},\begin{bmatrix}
B_{11}&B_{12}\\B_{21}&S
\end{bmatrix}\right):\begin{bmatrix}
H^2(\cE)\\ \cF
\end{bmatrix}\to
\begin{bmatrix}
H^2(\cE)\\ \cF
\end{bmatrix}.
$$
Now use these matrix representations and equate the (12)-entries of the relation $AT=TA$ to get $A_{12}W=M_zA_{12}$. Therefore $A_{12}=0$ by Lemma \ref{L:zero}. 
Similarly, from the relation $BT=TB$ we have $B_{12}=0$. Compare the (21)-entries of the relation $A=B^*T$ to get $A_{21}=0$. The same treatment for the relation 
$B=A^*T$ gives $B_{21}=0$.  Therefore we are left with the following relations
\begin{align*}
&A_{11}M_z=M_zA_{11}, \, B_{11}M_z=M_zB_{11},\, A_{11}=B_{11}^*M_z \text{ (and $B_{11} = A_{11}^* M_z$);}  \\
& RW=WR, \, SW=WS, \,  R=S^*W \text{ (and $S = W R^*$).}
\end{align*}
The second set of the above relations together with the fact that $W$ is a unitary implies that $(R,S,W)$ is a pseudo-commutative tetrablock unitary.  The first two 
intertwining relations in the first set implies that 
there exist bounded analytic functions $\Phi,\Psi:\mathbb D\to \cB(\cE)$ such that $A_{11} = M_\Phi$, $B_{11} = M_\Psi$.  The remaining relations in the first set
then give us
$$
 M_\Phi =   M_\Psi^* M_z, \quad M_\Psi = M_\Phi^* M_z.
$$
There now only remains a tedious computation with the power series expansions of $\Phi$ and $\Psi$ to see that the remaining  relations in the second set forces
$\Phi$ and $\Psi$ to have the coupled linear forms 
$$
\Phi(z)=G_1^*+zG_2 \quad\mbox{and}\quad \Psi(z)=G_2^*+zG_1
$$ 
for some operators $G_1,\, G_2 \in \cB(\cE)$.  Again the relation \eqref{PencilContr} is equivalent to $M_\Phi$ being a contraction operator.  
From Definition \ref{D:pc} it follows that $(M_\Phi, M_\Psi, M_z)$ is a pseudo-commutative tetrablock isometry. 
The completes the proof for the pseudo-commutative setting.

Suppose now that $(A, B, T)$ is a tetrablock isometry.  Then in particular $(A,B,T)$ satisfies all the requirements to be a pseudo-commutative tetrablock isometry,
so all the preceding analysis applies.  We then see that $(A,B,T)$ is unitarily equivalent to the triple in \eqref{Pencils}.  As $(A,B,T)$ now is
actually a tetrablock isometry, we have that $AB  = BA$.  The unitary equivalence then forces 
$$
\begin{bmatrix} M_{G_1^* + z G_2} & 0 \\ 0 & R \end{bmatrix} \begin{bmatrix} M_{G_2^* + z G_1} & 0 \\ 0 & S \end{bmatrix}
=  \begin{bmatrix} M_{G_2^* + z G_1} & 0 \\ 0 & S \end{bmatrix}  \begin{bmatrix} M_{G_1^* + z G_2} & 0 \\ 0 & R \end{bmatrix} 
$$
which can be split up as two commutativity conditions
\begin{align}
& M_{G_1^* + z G_2}  M_{G_2^* + z G_1} =      M_{G_2^* + z G_1}  M_{G_1^* + z G_2}   \label{com=1} \\
& RS = SR.  \label{com=2}
\end{align}
By reversing the computations done in item (1) of Example \ref{E:model}, we see that the intertwining \eqref{com=1} forces the set of conditions \eqref{com=}
Moreover, the condition \eqref{com=2} is exactly the missing ingredient needed to promote $(R,S,W)$ from a pseudo-commutative tetrablock unitary
to a strict tetrablock unitary.  This completes the proof.
\end{proof}

\begin{remark}  \label{R:semi-strict}  It will be useful to have a terminology for an intermediate class of operator triples $(A,B,T)$ which sits between
strict tetrablock isometries and general pseudo-commutative tetrablock isometries.  Let us say that the triple $(A,B,T)$ is a
\textbf{semi-strict tetrablock isometry} if $(A,B,T)$ is a pseudo-commutative isometry with Wold decomposition as in \eqref{Pencils} is
such that the pseudo-commutative tetrablock unitary component $(R,S,W)$ is actually a strict tetrablock unitary, i.e., $R$ and $S$ are contractions
which commute with each other as well as with $W$ which is unitary.
\end{remark}

We next present an analogue of the single-variable operator theory fact that any isometry can always be extended to a unitary.

\begin{corollary}  \label{C:pc} Pseudo-commutative/strict tetrablock isometries can be extended to pseudo-commutative/strict tetrablock unitaries.  More precisely:

\smallskip
\noindent
{\rm (1)}  A triple $(A,B,T)$ is a pseudo-commutative tetrablock isometry if and only if it extends to a pseudo-commutative tetrablock unitary. Moreover there exists an extension that acts on the space of minimal unitary extension of the isometry $T$.

\smallskip

\noindent
{\rm (2)} A triple $(A,B,T)$ is a  strict tetrablock isometry if and only if it extends to a strict tetrablock unitary acting on the space of the minimal unitary extension
of the isometry $T$.
\end{corollary}

\begin{proof}
If a triple extends to a pseudo-commutative tetrablock unitary, then from Definition \ref{D:pc} we can read off that it is also a pseudo-commutative tetrablock isometry. 
Similarly, if a triple extends to a strict tetrablock unitary, we can read off from criterion (iii) or (iv) in Theorem \ref{T:IsoChar} that the triple itself must be
a strict tetrablock isometry.

We now address the converse.   
In view of Theorem \ref{T:WoldPC}, we can assume without loss of generality that a pseudo-commutative tetrablock isometry $(A,B,T)$ is given in the form:
$$
\left(\begin{bmatrix}M_{G_1^*+z G_2}&0\\0&R \end{bmatrix},\begin{bmatrix}M_{G_2^*+z G_1}&0\\0&S \end{bmatrix},\begin{bmatrix}M_z &0\\0&W \end{bmatrix}\right): \begin{bmatrix}H^2(\cE)\\\cF\end{bmatrix}\to \begin{bmatrix}H^2(\cE)\\\cF\end{bmatrix}
$$
for some operators $G_1$, $G_2$ on $\cE$ and for some pseudo-commutative tetrablock unitary $(R,S,W)$ acting on $\cF$. Consider $H^2(\cE)\oplus\cF$ 
as a subspace of $L^2(\cE)\oplus\cF$ in the natural way. Then the triple
$$
\left(\begin{bmatrix}M_{G_1^*+\zeta G_2}&0\\0&R \end{bmatrix}, \begin{bmatrix}M_{G_2^*+\zeta G_1}&0\\0&S \end{bmatrix},
\begin{bmatrix}M_\zeta &0\\0&W \end{bmatrix}\right)\colon
 \begin{bmatrix}L^2(\cE)\\ \cF\end{bmatrix} \to \begin{bmatrix}L^2(\cE)\\ \cF\end{bmatrix}
$$
is an extension of $(A,B,T)$. The unitary $M_\zeta\oplus W$ is clearly a minimal unitary extension of the isometry $M_z\oplus W$. And since $M_{G_1^*+\zeta G_2}=M_{G_2^*+\zeta G_1}^*M_\zeta$ and $(R,S,W)$ is a pseudo-commutative tetrablock unitary, the above triple is a pseudo-commutative tetrablock unitary by Definition \ref{D:pc}.

If we start with a strict tetrablock isometry, then we shall also have that
$\sbm{ M_{G_1^*+ z G_2} & 0 \\ 0 & R }$ commutes with $\sbm{M_{G_2^* + z G_1} & 0 \\ 0 & S}$ on $\sbm{ H^2(\cE)  \\ \cF}$, or equivalently,
$RS = SR$ and the Toeplitz operator symbols equal to the pencils $G_1^* + z G_2$ and $G_2^* + z G_1$ commute:
$$
  (G_1^* + z G_2)(G_2^* + z G_1 ) = (G_2^* + z G_1 ) (G_1^* + z G_2).
 $$
 But then it is straightforward to see that this implies the commutativity of the associated Laurent operators acting on $L^2(\cE)$:
 $$
   M_{ G_1^* + \zeta G_2} M_{G_2^* + \zeta G_1} =   M_{G_2^* + \zeta G_1} M_{ G_1^* + \zeta G_2} .
 $$
 and hence also
 $$
 \begin{bmatrix} M_{ G_1^* + \zeta G_2}  & 0 \\ 0 & R \end{bmatrix}  \begin{bmatrix} M_{G_2^* + \zeta G_1} & 0 \\ 0 & S \end{bmatrix} =
          \begin{bmatrix}  M_{G_2^* + \zeta G_1} & 0 \\ 0 & S  \end{bmatrix}  \begin{bmatrix} M_{ G_1^* + \zeta G_2}  & 0 \\ 0 & R \end{bmatrix}. 
 $$
 on $\sbm{ L^2(\cE) \\ \cF}$.
 Moreover the extension of $M_{ G_1^* + z G_2}$ on $H^2(\cE)$ to $M_{ G_1^* + \zeta G_2}$ on $L^2(\cE)$ is norm-preserving, and hence the latter
 is contractive whenever the former is contractive, and similarly for $M_{G_2^* + z G_1}$ and $M_{G_2^* + \zeta G_1}$.
 We now have enough observations to conclude by criterion (iii) or (iv) in Theorem \ref{T:IsoChar} that 
 $\big( \left[ \begin{smallmatrix} M_{ G_1^* + \zeta G_2}  & 0 \\ 0 & R \end{smallmatrix} \right], 
 \left[ \begin{smallmatrix} M_{G_2^* + \zeta G_1} & 0 \\ 0 & S \end{smallmatrix} \right] , \left[ \begin{smallmatrix} M_\zeta & 0 \\ 0 & W \end{smallmatrix} \right] \big)$
on $\sbm{ L^2(\cE) \\ \cF}$ is a  tetrablock unitary as required.
 \end{proof}
 
 Another one-variable fact is the result due to Sz.-Nagy-Foias (see \cite[Theorem I.3.2]{Nagy-Foias}):  {\em any contraction operator $T$ on a Hilbert space $\cH$
  can be decomposed as $T = \sbm { T_{cnu} & 0 \\ 0 & U}$ where $T_{cnu}$ is a {\em completely nonunitary} (c.n.u.) contraction operator (meaning there is no 
  reducing subspace of $\cH$ such that $T|_\cH$ is unitary) and where $U$ is unitary}.  There is a analogous result for the setting of tetrablock unitaries
  and tetrablock contractions.  We say that the tetrablock contraction $(A,B,T)$ on $\cH$  is a \textbf{c.n.u.~tetrablock contraction} if there is no nontrivial
  jointly reducing subspace $\cH_u \subset \cH$ for $(A,B,T)$ such that $(A,B,T)|_{\cH_u}$ is a tetrablock unitary.  The following result appears in
  \cite{PalJMAA16}
  
  \begin{thm} \label{T:tetra-can-decom}  Let $(A,B,T)$ be a tetrablock contraction on $\cH$.  Then $\cH$ has an internal orthogonal direct-sum decomposition 
  $\cH = \cH_{c.n.u.} \oplus \cH_u$ with $\cH_{c.n.u.}$ and $\cH_u$ jointly reducing for $(A,B,T)|_{\cH_{c.n.u.}}$ equal to a c.n.u.\ tetrablock contraction
  and $(A,B,T)|_{\cH_u}$ equal to a tetrablock unitary.
  \end{thm}
  
  The at first surprising fact is that the same decomposition $\cH = \cH_{c.n.u.} \oplus \cH_u$ inducing the canonical decomposition of the contraction operator
  $T$ into is c.n.u.~part $T_{c.n.u.}$ and its unitary part $T_u$  turns out to also be jointly reducing for the whole operator triple $(A,B,T)$ and induces
  the canonical  tetrablock decomposition for the tetrablock contraction $(A,B,T)$, i.e., 
  $(A,B,T)|_{\cH_{c.n.u.}}$ is a c.n.u.\ tetrablock contraction, and $(A,B,T)|_{\cH_u}$ is a tetrablock unitary.
  
  \begin{remark}  \label{R:tetra-can-decom}    The import of Theorem \ref{T:tetra-can-decom} for model theory is the same as is the case for the classical case:
  since the model theory for a (strict) tetrablock unitary  is already well understood (see item (2) in Remark \ref{R:pseudocom}, it follows that it is perfectly
  satisfactory to focus on the case where $(A,B,T)$ is a c.n.u.~tetrablock contraction for the purposes of model theory.
  \end{remark}

\subsection{A canonical construction of a tetrablock unitary from a tetrablock contraction}\label{SS:canonical}

In this section we start with a tetrablock contraction $(A,B,T)$ and construct a tetrablock unitary which is uniquely associated with $(A,B,T)$ in a sense that will be made precise later in this subsection. This will be used later to construct a concrete functional model for a  pseudo-commutative tetrablock-isometric lift for $(A,B,T)$
which can be viewed as a functional model for $(A,B,T)$ itself.

We start by using the fact that the last entry $T$ in our tetrablock contraction $(A, B, T)$ is a contraction operator. Hence there exist a positive semidefinite operator $Q_{T^*}$ such that
\begin{align}\label{Q}
Q_{T^*}^2:=\operatorname{SOT-}\lim T^nT^{* n}.
\end{align}
Define the operator $X_{T^*}^*:\overline{\operatorname{Ran}}\;Q_{T^*}\to\overline{\operatorname{Ran}}\;Q_{T^*}$ densely by
\begin{align}\label{theX}
X_{T^*}^*Q_{T^*}h=Q_{T^*}T^*h.
\end{align}
This is an isometry because for all $h\in\cH$,
\begin{align} 
& \| X_{T^*}^* Q_{T^*}h\|^2  = \langle Q_{T^*}^2 T^*h, T^*h\rangle=\lim_{n\to\infty}\langle T^{n}T^{* n}T^*h,T^*h\rangle  \notag \\
& \quad = \lim_{n \to \infty} \langle T^{n+1} T^{*(n+1)} h, h \rangle  =\langle Q_{T^*}^2 h,  h\rangle = \| Q_{T^*}  h \|^2.
\label{Xisoprf}
\end{align}
Since $A$ is a contraction, we have for all $h\in\cH$
\begin{align*}
\langle AQ_{T^*}^2A^*h,h\rangle=\langle \lim_n T^nAA^*T^{*n} h,h\rangle \leq \lim_n\langle T^nT^{*n}h,h\rangle=\langle Q_{T^*}^2h,h \rangle.
\end{align*}
The same computation for the contraction $B$ will yield the same inequality involving $B$ in place of $A$. Consequently, the operators 
$A_{T^*}.B_{T^*}:\overline{\operatorname{Ran}}\;Q_{T^*}\to\overline{\operatorname{Ran}}\;Q_{T^*}$ defined densely by
\begin{align}\label{AB}
A_{T^*}^*Q_{T^*}h=Q_{T^*}A^*h \quad\text{and}\quad B_{T^*}^*Q_{T^*}h=Q_{T^*}B^*h
\end{align}
are contractions, and extend contractively to all of $\overline{\operatorname{Ran}}\, Q_{T^*}$ by a limiting process.
 Furthermore, from the definitions it is easy to see that $(A_{T^*}, B_{T^*}, X_{T^*})$ is a commutative triple since
by assumption we know that $(A^*, B^*, T^*)$ is commutative.
This and \eqref{theX} imply that if $f$ is a three-variable polynomial, then for all $h\in\cH$,
\begin{align*}
& \| f(A_{T^*}^*,B_{T^*}^*,X_{T^*}^*) Q_{T^*} h\| = \| Q_{T^*} f(A^*,B^*,T^*) h\| \\
& \quad \le \| f(A^*, B^*, T^*) \|  \| h || \leq \| f(A^*,B^*,T^*)\| \| h \|  \leq   \big(  \sup_{\mathbb E} |f| \big) \| h \|
\end{align*}
where in the last inequality we used the fact that $(A,B,T)$ is a tetrablock contraction. This inequality together with the fact that $X_{T^*}^*$ is an isometry implies that 
$(A_{T^*}^*,B_{T^*}^*,X_{T^*}^*)$ is a tetrablock isometry. By Corollary \ref{C:pc}, $(A_{T^*}^*,B_{T^*}^*,X_{T^*}^*)$ has a tetrablock unitary extension 
$(R_D^*,S_D^*,W_D^*)$ acting on  a space which we shall call $\cQ_{T^*} \supseteq \overline{\operatorname{Ran }}\;Q_{T^*}$, 
where $W_D^*$ acting on $\cQ_{T^*}$ is the minimal unitary extension of $X_{T^*}^*$. 

\begin{definition}\label{D:canonical}
Let $(A,B,T)$ be a tetrablock contraction and let $(R_D,S_D,W_D)$ be the tetrablock unitary constructed from $(A,B,T)$ as above. We say that $(R_D,S_D,W_D)$ 
is the {\bf canonical tetrablock unitary} associated with the tetrablock contraction $(A,B,T)$.
\end{definition}

The next result assures us that canonical tetrablock unitaries associated with the same tetrablock contraction $(A,B,T)$ are the same up to unitary equivalence.

\begin{thm}\label{T:UniCanon}
Let $(A,B,T)$ on $\cH$ and $(A',B',T')$ on $\cH'$ be two tetrablock contractions with $(R_D,S_D,W_D)$ and $(R_D',S_D',W_D')$ equal to the respective canonical 
tetrablock unitaries. If $(A,B,T)$ and $(A',B',T')$ are unitarily equivalent via $\tau$, then $(R_D,S_D,W_D)$ and $(R_D',S_D',W_D)$ are unitarily equivalent via 
$\omega_\tau:\cQ_{T^*}\to\cQ_{T^{'*}}$
\begin{align}\label{omega-tau}
\omega_\tau: W_D^n Q_{T^*}h\mapsto W_D^{'n} Q_{T'^*}\tau h 
\end{align}for all $n\geq0$ and $h\in\cH$.
\end{thm}

\begin{proof}
Let the spaces $\cQ_{T^*}, \cQ_{T^{'*}}$ and the operators $\{A_{T^*},B_{T^*},Q_{T^*}\}$, $\{A_{T'^*},B_{T'^*},Q_{T'^*}\}$ be obtained as above from $(A,B,T)$ and 
$(A',B',T')$, respectively. Since $\tau$ is a unitary intertwining $T$ and $T'$, it intertwines $T^*$ and $T'^*$ and thus $\tau Q_{T^*}=Q_{T'^*} \tau$. 
Therefore by definition \eqref{AB} it follows that $\tau(A_{T^*},B_{T^*},Q_{T^*})=(A_{T'^*},B_{T'^*},Q_{T'^*})\tau$.
By definition of $\omega_\tau$ it is clear that $\omega_\tau W_D=W_D'\omega_\tau$. Therefore for every $h\in\cH$ and $n\geq 0$,
\begin{align*}
\omega_\tau R_DW_D^n Q_{T^*}h&=\omega_\tau W_D^{n+1}W_D^*R_DQ_{T^*}h\\
&=\omega_\tau W_D^{n+1}S_D^*Q_{T^*}h \quad \mbox{[using Theorem \ref{T:IsoChar}, part (iv)]}\\
&=W_D'^{(n+1)}\tau B_{T^*}^*Q_{T^*}h=W_D'^{(n+1)} B_{T'^*}^* Q_{T'^*}\tau h\\
&=W_D'^{(n+1)} S_D'^* Q_{T'^*}\tau h\\
&=S_D'^*W_D'^{(n+1)}Q_{T^*} \quad \mbox{[since $S_D'^*W_D'=W_D'S_D'^*$]}\\
&=R_D' W_D'^nQ_{T'^*}\tau h\quad \mbox{[using Theorem \ref{T:IsoChar}, part (iii)]}\\
&=R_D'\omega_\tau W_D^nQ_{T^*}h.
\end{align*}
A similar computation shows that $\omega_\tau S_D=S_D'\omega_\tau$.
\end{proof}

\subsection{The fundamental operators}\label{SS:FundOps}
Much of the theory of tetrablock contractions is heavily based on a pair of operators that is uniquely associated with a tetrablock contraction. These are called the 
fundamental operators, the existence of which was proved in \cite{Tirtha14} with appeal to connections between tetrablock contractions and symmetrized-bidisk
contractions.   We state the result and sketch a more self-contained proof with the appeal to symmetrized-bidisk theory eliminated.
In the sequel  we shall use the notation $\nu(X)$ to denote the numerical radius of the operator $X$ on the Hilbert space $\cH$:
$$
\nu(X):= \sup_{x \in \cH\colon \| x \| =1} | \langle X x, x \rangle_\cH |.
$$

\begin{thm}\label{T:FundTetra}
Let $(A,B,T)$ be a tetrablock contraction on a Hilbert space $\mathcal H$.

\begin{enumerate}
\item[(i)]{\rm (See \cite[Theorem 3.4]{Tirtha14})} There exist two unique operators $F_1$ and $F_2$ acting on $\cD_T$ with the numerical radii at most one such that
\begin{align}\label{FundEqns}
A-B^*T=D_TF_1D_T \quad\text{and}\quad B-A^*T=D_TF_2D_T.
\end{align}
Moreover, the operators $F_1,F_2$ are such that $\nu(F_1+zF_2)\leq1$ for all $z\in \overline{\mathbb{D}}$.

\item[(ii)]{\rm (See \cite[Corollary 4.2]{Tirtha14})} The operators $F_1,F_2$ are alternatively characterized as the unique bounded operators on 
$\cD_T$ such that $(X_1, X_2) = (F_1, F_2)$ satisfies the system of operator equations 
\begin{align}\label{Determining}
D_TA=X_1D_T+X_2^*D_TT \quad\text{and}\quad D_TB=X_2D_T+X_1^*D_TT.
\end{align}
\end{enumerate}
\end{thm}

\begin{proof}
Let $(A,B,T)$ be a tetrablock contraction on $\cH$. Since for every $z\in\bD$, $\Psi(z,\cdot)$ as in item (ii) of Theorem \ref{T:CharacTetra} is analytic in an open set 
containing $\overline{\bE}$, and $\overline{\bE}$ is polynomially convex, a limiting argument  implies that $\Psi(\zeta,(A,B,T))$ is a contraction for every $\zeta\in\bT$,
 or equivalently, on simplifying $I-\Psi(\zeta,(A,B,T))^*\Psi(\zeta,(A,B,T))\succeq 0$ we get
\begin{align*}
    (I-T^*T)+(B^*B-A^*A) - \zeta(B-A^*T) - \overline{\zeta}(B-A^*T)^*\succeq 0.
\end{align*}Similarly, applying item (iii) of Theorem \ref{T:CharacTetra} we have for every $\alpha\in\bT$,
\begin{align*}
    (I-T^*T)+(A^*A-B^*B) - \alpha(A-B^*T) - \overline{\alpha}
(A-B^*T)^*\succeq 0.
\end{align*}
On adding the above two positive operators and then simplifying we get
\begin{align}\label{Posit3}
    D_T^2 \succeq \operatorname{Re}\alpha \left[(A-B^*T)+\beta(B-A^*T)\right]
\end{align}for every $\alpha,\beta\in\bT$. 
We now make use the following  lemma of independent interest:

\begin{lemma}{\rm (See \cite[Lemma 4.1]{B-P-SR})}  \label{L:Tirtha14}
Let $\Sigma$ and $D$ be two operators such that
$$
DD^*\succeq\operatorname{Re}\alpha\Sigma \quad\mbox{for all }\alpha\in\bT.
$$
Then there exists an operator $F$ acting on $\overline{\operatorname{Ran}} \, D^*$ with numerical radius at most one such that $\Sigma=DFD^*$.
\end{lemma}

\begin{proof}[Sketch of proof.]  Apply the Fejer-Riesz factorization theorem of Dritschel-Rovnyak   \cite[Theorem 2.1]{DM-FRThm} to the Laurent operator-valued polynomial 
$P(e^{i\theta}) = 2DD^*-e^{i\theta}\Sigma-e^{-i\theta}\Sigma^*$.   Along the way one makes use of the standard Douglas lemma ($\exists$ $X \in \cB(\cH)$ with
$AX = B$ $\Leftrightarrow$ $B B^* \preceq A A^*$) and a criterion for a
Hilbert space operator to have numerical radius at most $1$:  $X \in \cB(\cH)$ has $\nu(X) \le 1$ $\Leftrightarrow$ $\operatorname{Re} \, ( \beta X)  
\preceq I_\cH$ for all $\beta \in {\mathbb T}$.  Note that Lemma \ref{L:Tirtha14} can itself be viewed as a quadratic, numerical-radius version of the Douglas lemma.
\end{proof}

We apply Lemma \ref{L:Tirtha14}  to the case \eqref{Posit3} for each $\beta$ to get a numerical contraction $F(\beta)$ such that
\begin{align}\label{FundEqnsBeta}
(A-B^*T)+\beta(B-A^*T)=D_TF(\beta)D_T.    
\end{align}
On adding equations \eqref{FundEqnsBeta} for the cases $\beta=1$ and $-1$, we get
\begin{align}\label{beta=1,-1}
    A-B^*T=D_TF_1D_T\quad\mbox{where}\quad F_1:=\frac{F(1)+F(-1)}{2}.
\end{align}
Thus putting $\beta=1$ in \eqref{FundEqnsBeta} and combining with \eqref{FundEqnsBeta} gives us
\begin{equation}   \label{beta=1}
B-A^*T=D_T(F(1)-F_1)D_T=D_TF_2D_T\quad\mbox{where}\quad F_2:=\frac{F(1)-F(-1)}{2}.
\end{equation}
We conclude that $F_1$ and $F_2$ so constructed satisfy equations \eqref{FundEqns}.  It is easy to see that in general
\begin{equation}  \label{easy}
X \in \cB(\cD_T) \text{ with } D_T X D_T = 0 \Rightarrow X = 0.
\end{equation}
Applying this observation to the homogeneous version of equations \eqref{FundEqns} implies that the solutions $(F_1, F_2)$
of \eqref{FundEqns} must be unique whenever they exist.

On the other hand, if we combine \eqref{beta=1,-1} with \eqref{beta=1} we see that 
$\widetilde F(\beta) := F_1 + \beta F_2$ gives us a second solution of \eqref{FundEqnsBeta}. 
Hence 
$$
D_T (\widetilde F(\beta) - F(\beta)) D_T = 0 \text{ where } \widetilde F(\beta) - F(\beta) \in \cB(\cD_T) \text{ for all } \beta \in {\mathbb T}.
$$
Again by \eqref{easy}, we see that
$\widetilde F(\beta) = F(\beta)$ for all $\beta \in {\mathbb T}$.  But we saw above (as a consequence of Lemma \ref{L:Tirtha14}) that
$F(\beta)$ is a numerical contraction for all $\beta \in {\mathbb T}$.  As we now know that $\widetilde F(\beta) = F(\beta)$, we conclude that
the pencil $\widetilde F(\beta) = F_1 + \beta F_2$ is a numerical contraction for all $\beta \in {\mathbb T}$.  By applying the Maximum Modulus Theorem to the
holomorphic function $\langle (F_1 + \beta F_2)h, h \rangle$ for each fixed $h \in \cH$, we see that $F_1 + z F_2$ is a numerical contraction
for all $z \in \overline{\mathbb D}$.  This completes the proof of item (i) in Theorem \ref{T:FundTetra}.

 To see  that $F_1$ and $F_2$ satisfy equations \eqref{Determining}, simply multiply $D_T$ on the left of each equation and use the identities \eqref{FundEqns}
 to simplify.   To show that $F_1,F_2$ are the unique operators on $\cD_T$  satisfying these two equations, it is enough to show that  $X= 0$ and $Y=0$ 
 are the only operators in $\cB(\cD_T)$ satisfying
$$
XD_T+Y^*D_TT=0, \quad YD_T+X^*D_TT=0.
$$
To show that $X=0$, compute
\begin{align*}
    D_TXD_T=-D_TY^*D_TT=T^*D_TXD_TT&=-T^*D_TY^*D_TT^2\\
    &=T^{*2}D_TXD_TT^2.
\end{align*}
Thus by iteration of the above process $D_TXD_T=T^{*n}D_TXD_TT^n$. This shows that $X=0$ because for every $h\in\cH$
$$
\lim_n\|D_TT^nh\|^2=\lim_n\|T^nh\|^2-\lim_n\|T^{n+1}h\|^2=0.
$$
A similar argument gives $Y=0$. This is the idea of the proof due by Bhattacharyya \cite{Tirtha14}.
\end{proof}

The unique operators $F_1,F_2$ in item (i) of Theorem \ref{T:FundTetra} will be referred to as the {\bf fundamental operators} for the tetrablock contraction $(A,B,T)$, as in \cite{Tirtha14}.

\smallskip
As we have seen in Proposition \ref{P:symmetries}, if $(A,B,T)$ is a ${\mathbb E}$-contraction, so also is $(A^*, B^*, T^*)$.
For the construction of the functional model for a tetrablock contraction $(A,B,T)$, 
it turns out to be more convenient to work with the fundamental operators for the 
adjoint tetrablock contraction $(A^*, B^*, T^*)$ which we denote as $(G_1, G_2)$.  Thus there exists exactly one solution 
$(X_1, X_2)  = (G_1, G_2)$ of the system of equations
\begin{equation}   \label{FundEquns*}
A^* - B T^* = D_{T^*} X_1 D_{T^*}, \quad B^* - A T^* = D_{T^*} X_2 D_{T^*}
\end{equation}
with equivalent characterization as the unique solution $(X_1, X_2) = (G_1, G_2)$ of the second system of equations
\begin{equation}  \label{Determining*}
D_{T^*} A^* = X_1 D_{T^*} + X_2^* D_{T^*} T^*, \quad D_{T^*} B^* = X_2 D_{T^*} + X_1^* D_{T^*} T^*
\end{equation}
Then from Example \ref{E:model} we see immediately that the operator pair $(M_{G_1^* + z G_2}, M_{G_2^* + z G_1}, M_z)$ acting on the Hardy space 
$H^2(\cD_{T^*})$ is of the correct  form to be 
a pseudo-commutative tetrablock isometry. We would like to establish conditions under which this a priori only pseudo-commutative ${\mathbb E}$-isometry
is actually a strict ${\mathbb E}$-isometry.  This follows from the following result.

\begin{thm}  \label{T:pseudo-strict} {\rm (See \cite{Tirtha14}.)}  Suppose that $(G_1, G_2)$ is the Fundamental Operator pair for the ${\mathbb E}$-contraction 
$(A^*, B^*, T^*)$.  Let $(V_1, V_2, V_3)$ be the operator triple
\begin{equation}   \label{Eisom-candidate}
 (V_1, V_2, V_2) = (M_{G_1^* + z G_2}, M_{G_2^* + z G_1}, M_z) \text{ on } H^2(\cD_{T^*}).
\end{equation}
 Then:
 
 \begin{enumerate}
 \item $(V_1, V_2, V_3)$ is a pseudo-commutative tetrablock isometry having the additional property that
 $$
   r(V_1) \le 1, \quad r(V_2) \le 1.
 $$
 
\item  Suppose in addition that the Fundamental Operator pair $(G_1, G_2)$ satisfy the commutativity conditions.  Then
$(V_1, V_2, V_3)$ is a strict tetrablock isometry. 
 \end{enumerate}
 \end{thm}
 
 \begin{corollary} \label{C:strict} Let $(G_1, G_2)$ be the Fundamental Operator pair for the tetrablock contraction $(A^*, B^*, T^*)$ and set $(V_1, V_2, V_3)$
 as in \eqref{Eisom-candidate}.  Then $(V_1, V_2, V_3)$ is a (strict) tetrablock isometry if and only if the commutativity conditions  \eqref{com=} hold.
 \end{corollary}
 
 \begin{proof}[Proof of Corollary \ref{C:strict}]  If the commutativity conditions \eqref{com=} are satisfied, then statement (2) of Theorem \ref{T:pseudo-strict}
 says that $(V_1, V_2, V_3)$ is a tetrablock isometry.  Conversely, if $(V_1, V_2, V_3)$ is a tetrablock isometry, in particular $(V_1, V_2, V_3)$ is
 commutative so the commutativity conditions \eqref{com=} are satisfied.
  \end{proof}

\begin{proof}[Proof of Theorem \ref{T:pseudo-strict}]
We first consider statement (1).  That $(V_1, V_2, V_3)$ is a pseudo-commutative tetrablock isometry follows from the fact that it has the required form \eqref{pcform}
as presented in Example \ref{E:model}.  It remains to use the fact that $(G_1, G_2)$ is a Fundamental Operator pair for the tetrablock contraction $(A^*, B^*, T^*)$
(to see that we also have $r(V_1) \le 1$ and $r(V_2) \le 1$.  

By Theorem \ref{T:FundTetra} (applied to $(A^*, B^*, T^*)$ in place of $(A,B,T)$), we know that
$\nu(G_1 + z G_2) \le 1$ for all  $z \in \overline{\mathbb D}$.  By the And\^o criterion for the numerical radius of an operator to be no more than 1, this means that
$$
\beta (G_1 + \alpha G_2) + \overline{\beta} (G_1^* + \overline{\alpha} G_2^*) \preceq 2 I_{\cD_{T^*}} \forall \alpha, \beta \in \overline{\mathbb D}.
$$
Rearrange this inequality as
\begin{align*}
& (\overline{\beta} G_1^* + \beta \alpha G_2) + (  \beta G_1 + \overline{\beta}  \overline{\alpha} G_2^*)  \\
& \quad = \overline{\beta} (G_1^* + \beta^2 \alpha G_2) + \beta (G_1 + \overline{\beta}^2 \overline{\alpha} G_2^*) \preceq 2I_{\cD_{T^*}}.
\end{align*}
By the And\^o criterion applied in the reverse direction, this tells us that
$$
   \nu(G_1^* + z G_2) \le 1 \text{ for all } z \in \overline{\mathbb D}.
$$
But in general the numerical radius dominates the spectral radius; thus
$$
  r(G_1^* + z G_2) \le \nu(G_1^* + z G_2) \le 1 \text{ for all } z \in \overline{\mathbb D}.
$$
If we choose $\lam \in {\mathbb C}$ with $|\lam| > 1$, then $\lam I_{\cD_{T^*}} - (G_2^* + z G_1)$ is invertible, and in fact
the operator-valued function $z \mapsto (\lam I_{\cD_{T^*}} - (G_2^* + z G_1))^{-1}$ is in $H^\infty(\cB(\cD_{T^*})$.  We thus conclude that in fact
$$
  r(M_{G_2^* + z G_1}) \le 1.
$$

All the above analysis applies to the pair $(G_2, G_1)$ in place of $(G_1, G_2)$, as $(G_2, G_1)$ is the Fundamental Operator pair for the ${\mathbb E}$-contraction
$(B^*, A^*, T^*)$;  hence we also have also
$$
 r(M_{G_2^* + z G_1}) \le 1.
$$
This completes the proof of statement (1).

We now consider statement (2).  As we are now assuming that  $(G_1, G_2)$ satisfy the commutativity conditions \eqref{com=},  it follows that
$M_{G_1^* + z G_2}$ commutes with $M_{G_2^* + z G_1}$.
We are now in a position to apply statement (3) in Remark \ref{R:pseudocom} (or, what is the same,  statement (2) in Theorem \ref{T:pc-vs-strict-Eisom}
to conclude that the a priori only pseudo-commutative ${\mathbb E}$-contraction
$ (M_{G_1^* + z G_2}, M_{G_2^* + z G_1},$ $ M_z)$ is in fact a strict ${\mathbb E}$-contraction.
\end{proof}

\section{Functional models for tetrablock contractions}\label{S:fmTetr}
In this section we produce two functional models for tetrablock contractions, the first  inspired by model theory of Douglas 
\cite{Doug-Dilation}, and the second  by the model theory of Sz.-Nagy and Foias \cite{Nagy-Foias}. We note that so far only a functional model is 
known for the special case when the last entry is a pure contraction; see \cite[Theorem 4.2]{SauNYJM}.

\subsection{A Douglas-type functional model}\label{SS:fmTet}
Let $T$ be any contraction on $\cH$. Define the operators $\cO_{D_{T^*}, T^*}:\cH\to H^2(\cD_{T^*})$ as
\begin{align}\label{observ}
 \cO_{D_{T^*}, T^*}(z) h = \sum_{n=0}^\infty z^nD_{T^*} T^{*n} h, \text{ for every }h\in\cH,
\end{align}
and $\Pi_D:\cH\to\sbm{H^2(\cD_{T^*})\\ \cQ_{T^*}}$ by
\begin{align}\label{Pi-D}
\Pi_Dh=
\begin{bmatrix}
 \cO_{D_{T^*}, T^*}(z) h \\ Q_{T^*}h\end{bmatrix} \quad\text{for all } h\in\cH,
\end{align}
where $Q_{T^*}$ is as in \eqref{Q}. Then the computation
\begin{align*}
\|\Pi_D h\|^2=&\| \cO_{D_{T^*}, T^*}(z) h\|_{H^2(\cD_{T^*})}^2+ \|Q_{T^*}h\|^2\\
&=\sum_{n=0}^\infty \|D_{T^*} T^{*n} h\|^2+\lim_{n\to\infty}\|T^{*n}h\|^2\\
&=(\|h\|^2-\|T^*h\|^2)+(\|T^{*}h\|^2-\|T^{*2}h\|^2)+\cdots) +\lim_{n\to\infty}\|T^{*n}h\|^2\\
&=\|h\|^2
\end{align*}
shows that $\Pi_D$ is an isometry. Note the following intertwining property of $\cO_{D_{T^*}, T^*}$:
\begin{align}
 \cO_{D_{T^*}, T^*}(z)T^*h & =\sum_{n=0}^\infty z^nD_{T^*} T^{*n+1} h=M_z^*\sum_{n=0}^\infty z^nD_{T^*} T^{*n} h \notag  \\
& =M_z^* \cO_{D_{T^*}, T^*}(z)h.
\end{align}
This together with intertwining \eqref{theX} of $Q_{T^*}$ implies
\begin{align}\label{lift}
\Pi_DT^*=
\begin{bmatrix}
M_z&0\\0&W_D 
\end{bmatrix}^*\Pi_D.
\end{align}
This shows that the the pair
$$
V_D:=\begin{bmatrix} M_z&0\\0&W_D \end{bmatrix}:   \begin{bmatrix} H^2(\cD_{T^*})\\ \cQ_{T^*} \end{bmatrix}
\to \begin{bmatrix} H^2(\cD_{T^*})\\ \cQ_{T^*} \end{bmatrix}
$$
is an isometric lift of $T$. This construction is by Douglas, where he also showed that this lift is minimal (see \cite{Doug-Dilation}).

Now let $(A,B,T)$ be a tetrablock contraction acting on $\cH$ and $G_1,G_2$ be the fundamental operators of $(A^*,B^*,T^*)$. 
Let $(R_D,S_D,W_D)$ acting on $\cQ_{T^*}$ be the canonical tetrablock unitary associated with $(A,B,T)$. Consider the operators
\begin{align}
\left(\begin{bmatrix}
M_{G_1^*+zG_2}&0\\0&R_D
\end{bmatrix},\begin{bmatrix}
M_{G_2^*+zG_1}&0\\0&S_D
\end{bmatrix},\begin{bmatrix}
M_z&0\\0& W_D
\end{bmatrix}\right)
\mbox{ on }\begin{bmatrix}
H^2(\cD_{T^*})\\ \cQ_{T^*}
\end{bmatrix}.
\end{align}
We claim that 
\begin{align}\label{claim}
\notag &\Pi_D(A^*,B^*,T^*)\\
&=
\left(\begin{bmatrix}
M_{G_1^*+zG_2}&0\\0&R_D
\end{bmatrix}^*,\begin{bmatrix}
M_{G_2^*+zG_1}&0\\0&S_D
\end{bmatrix}^*,\begin{bmatrix}
M_z&0\\0& W_D
\end{bmatrix}^*\right)\Pi_D,
\end{align}
where $\Pi_D:\cH\to \sbm{H^2(\cD_{T^*})\\ \cQ_{T^*}}$ is the isometry as in \eqref{Pi-D}. 

Recalling the definition $\Pi_D = \sbm{ \cO_{D_{T^*}, T^*} \\ \cQ_{T^*}} \colon \cH \to \sbm{H^2(\cD_{T^*}) \\ \cQ_{T^*}}$, we see that the three-fold
intertwining condition \eqref{claim} splits into two three-fold intertwining conditions
\begin{align}
& \cO_{D_{T^*}, T^*} (A^*, B^*, T^*) =  (M^*_{G_1^* + z G_2}, M^*_{G_2^* + z G_1}, M^*_z) \cO_{D_{T^*}, T^*}, \label{claim1} \\
&  Q_{T^*} (A^*, B^*, T^*) = (R^*_D, S^*_D, W^*_D) Q_{T^*}.  \label{claim2}
\end{align}
The last equation in \eqref{claim1} combined with the last equation in \eqref{claim2} we have already seen as the condition that $\Pi_D$ is the isometric 
identification map implementing $\sbm{ M_z & 0 \\ 0 & W_D}$ as the Douglas minimal isometric lift of $T$ (see \eqref{lift}).  We shall next check only the first equation
in \eqref{claim1} and the first equation in \eqref{claim2} as the verification of the respective second equations is completely analogous.  Thus it remains to check
\begin{align}
& \cO_{D_{T^*}, T^*} A^* = M^*_{G_1^* + z  G_2} \cO_{D_{T^*}, T^*},  \label{claim1'} \\
& Q_{T^*} A^* = R^*_D Q_{T^*}.  \label{claim2'}
\end{align}
Note that \eqref{claim2'} is part of the construction of the canonical tetrablock unitary associated with the original tetrablock contraction $(A,B,T)$ (see \eqref{AB}).
To check \eqref{claim1'}, let us rewrite the condition in function form:
$$
D_{T^*} (I - z T^*)^{-1} A^*  = G_1 D_{T^*} (I - z T^*)^{-1} + G_2^* D_{T^*}(I - z T^*)^{-1} T^*
$$
As $A$ commutes with $T$, we can rewrite this as
$$
D_{T^*} A^* (I - z T^*)^{-1} = (G_1 D_{T^*} + G_2^* D_{T^*} T^*) (I - z T^*)^{-1}.
$$
We may now cancel off the resolvent term $(I - z T^*)^{-1}$ to get a pure operator equation
\begin{equation}  \label{need}
D_{T^*} A^* = G_1 D_{T^*} + G_2^* D_{T^*} T^*.
\end{equation}
Let us now recall that the operators $(G_1, G_2)$ on $\cD_{T^*}$ were chosen to be the Fundamental Operators for the tetrablock contraction
$(A^*, B^*, T^*)$.  Thus by our earlier discussion of Fundamental Operators for tetrablock contractions (see Theorem \ref{T:FundTetra}), we know that 
$(X_1, X_2) = (G_1, G_2)$ satisfies the identities \eqref{Determining*}, the first of which gives the same condition on $(G_1, G_2)$ as \eqref{need}.
 Thus this choice of $(G_1, G_2)$ indeed leads to a solution of \eqref{claim1'}, and the proof of \eqref{claim} is complete.  (Of course the second equation
 in \eqref{Determining*} amounts to the verification of the second equation in \eqref{claim1}.)

This then means that the operator triple with isometric embedding operator $\Pi_D$
\begin{align} 
& (V_1, V_2, V_3) :=   \notag \\
& \quad\bigg(\Pi_D, \bigg( \begin{bmatrix} M_{G_1^* + z G_2} & 0 \\ 0 & R_D \end{bmatrix}, 
\begin{bmatrix} M_{G_2^* + z G_1} & 0 \\ 0 & S_D \end{bmatrix}, \begin{bmatrix}  M_z & 0 \\ 0 & W_D \end{bmatrix}  \bigg)\bigg)  \label{candidate}
\end{align}
is a lift of the tetrablock contraction $(A,B,T)$.  As $(R_D, S_D, W_D)$ is ${\mathbb E}$-unitary as part of the canonical construction in Section
\ref{SS:canonical}, we see from the form of the top components of $(V_1, V_2, V_3)$ in \eqref{candidate} that $(V_1, V_2, V_3)$ is a semi-strict
${\mathbb E}$-isometry and is a strict ${\mathbb E}$-isometry exactly when the top-component triple
$( M_{G_1^* + z G_2}, M_{G_2^* + z G_1}, M_z)$ is a strict ${\mathbb E}$-isometry.  By Theorem \ref{T:pseudo-strict}, this in turn happens exactly
when the Fundamental Operator pair $(G_1, G_2)$ satisfies the commutativity conditions \eqref{com=}.
In summary we have verified most of the following result.  We note that item (2) recovers a result of Bhattacharyya-Sau \cite{BS-CAOT}
via functional-model methods rather than by block-matrix-construction methods.

\begin{thm}\label{T:Dmodel}
{\rm (1)}
Let $(A,B,T)$ be a tetrablock contraction on $\cH$ and let $V_3$ on $\cK \supset \cH$ be the (essentially unique) minimal isometric lift for the 
contraction operator $T$.
Then there is a unique choice of operators $(V_1, V_2)$ on $\cK$ so that the triple $\bfV = (V_1, V_2, V_3)$ is a semi-strict tetrablock isometric lift for $(A,B,T)$.

\smallskip

\noindent
{\rm (2)} A necessary and sufficient condition that there be a strict tetrablock isometric lift
$\bfV = (V_1, V_2, V_3)$ for $(A,B,T)$ with the isometry $V_3$ equal to a minimal isometric lift for $T$ is that the fundamental operators $(G_1, G_2)$ for the adjoint tetrablock contraction $(A^*, B^*, T^*)$ satisfy the system of operator equations \eqref{com=}.  In this case the operator pair $(V_1, V_2)$ on $\cK$ 
is uniquely determined once one fixes a choice (essentially unique) for a minimal isometric lift $V_3$   for $T$.

\smallskip

\noindent
{\rm (3)}
The lift $(V_1, V_2, V_3)$ can be given in functional form in the coordinates of the Douglas model as follows.  
For $(A,B,T)$ equal to  a tetrablock contraction on a Hilbert space $\mathcal{H}$ let $(G_1,G_2)$ be the fundamental operators of $(A^*,B^*,T^*)$,
let $(R_D, S_D, W_D)$ be the tetrablock unitary canonically associated with $(A,B,T)$ as in  Definition \ref{D:canonical},
and let
$\Pi_D = \sbm{ \cO_{D_{T^*}, T^*} \\ Q_{T^*}}$ be the Douglas isometric embedding map from $\cH$ into $\sbm{ H^2(\cD_{T*}) \\  \cQ_{T^*} }$.
Then 
\begin{equation}  \label{unique-lift}
\bigg(  \Pi_D, \bigg( \begin{bmatrix} M_{G_1^* + z G_2} & 0 \\ 0 & R_D \end{bmatrix}, 
\begin{bmatrix} M_{G_2^* + z G_1} & 0 \\ 0 & S_D \end{bmatrix}, \begin{bmatrix}  M_z & 0 \\ 0 & W_D \end{bmatrix}  \bigg) \bigg)
\end{equation}
is a semi-strict (strict exactly when $(G_1, G_2)$ satisfies \eqref{com=}) tetrablock isometric lift for $(A,B,T)$.  In particular $(A,B,T)$ is  jointly unitarily equivalent to
\begin{align}\label{SemiModel}
    P_{\cH_D}\left(\begin{bmatrix} M_{G_1^*+zG_2}&0\\0&R_D \end{bmatrix},\begin{bmatrix} M_{G_2^*+zG_1}&0\\0&S_D\end{bmatrix}, 
\begin{bmatrix} M_z&0\\0& W_D \end{bmatrix}\right) \bigg|_{\cH_D},
\end{align}
where $\cH_D$ is the functional model space given by
\begin{align}\label{HD}
\cH_D:=\operatorname{Ran}\Pi_D\subset \begin{bmatrix} H^2(\cD_{T^*}) \\ \cQ_{T^*} \end{bmatrix}.
\end{align}
and any other semi-strict ${\mathbb E}$-isometric lift $(V'_1, V'_2, V'_3)$ with $V'_3 = \sbm{ M_z & 0 \\ 0 & W_D}$ on $\sbm{ H^2(\cD_{T^*}) \\ \cQ_{T^*}}$
is equal to \eqref{unique-lift}.
\end{thm}

\begin{proof}  The discussion preceding the statement of the theorem amounts to a proof of statement (3) in Theorem \ref{T:Dmodel}.    Statements (1) and (2)
apart from the uniqueness assertion amounts to a coordinate-free (abstract, model-free) interpretation of the results of statement (3).  It remains only
to discuss the uniqueness assertion in statements (1) and (2).  This can also be formulated in terms of the model as follows:
{\em Given a tetrablock contraction $(A,B,T)$, let $\Pi_D = \sbm{ \cO_{D_{T^*}, T^*} \\ Q_{T^*}}$ be the Douglas embedding map and let $\cK_D= \sbm{ H^2(\cD_{T^*})
\\ \cQ_{T^*}}$ be the Douglas minimal isometric lift space for $T$ with $V_D = \sbm{ M_z & 0 \\ 0 & W_D}$  on $\cK_D$ equal to the Douglas minimal
isometric lift for $T$.  Suppose that 
$$
 \tA = \begin{bmatrix} \tA_{11} & \tA_{12} \\ \tA_{21} & \tA_{22} \end{bmatrix}, \quad
 \tB = \begin{bmatrix} \tB_{11}  & \tB_{12} \\ \tB_{21} & \tB_{22} \end{bmatrix}
$$
 are two operators on $\cK_D = \sbm{ H^2(\cD_{T^*} \\ \cQ_{T^*} }$ such that
 $ \big(\tA, \tB, \sbm{ M_z & 0 \\ 0 & W_D}\big)$ is a pseudo-commutative tetrablock isometric lift for $T$.    Then necessarily}
 $$
\tA  =  \begin{bmatrix} M_{G_1^* + z G_2} & 0 \\ 0 & R_D \end{bmatrix} , \, \tB = \begin{bmatrix} M_{G_2^* + z G_1} & 0 \\ 0 & S_D \end{bmatrix}.
$$
{\em where the pair $(G_1, G_2)$ is equal to the pair of Fundamental Operators for the  tetrablock contraction $(A^*, B^*, T^*)$, and where
$(R_D, S_D, W_D)$ is the tetrablock unitary canonically associated with the tetrablock contraction $(A, B, T)$ as in Definition \ref{D:canonical}.}

To prove this model-theoretic reformulation of the uniqueness problem,  we proceed as follows.  We are given  first of all that the triple
\begin{align}\label{pcLift}
\bigg( \tA = \begin{bmatrix} \tA_{11} & \tA_{12} \\ \tA_{21} & \tA_{22} \end{bmatrix}, \quad
 \tB = \begin{bmatrix} \tB_{11}  & \tB_{12} \\ \tB_{21} & \tB_{22} \end{bmatrix}, \quad 
 \begin{bmatrix} M_z & 0 \\ 0 & W_D \end{bmatrix}  \bigg)
 \end{align}
 is a pseudo-commutative tetrablock isometry. According to Definition \ref{D:pc} we have the following:
 \begin{itemize}
 \item{(i)} $ \tA \sbm{ M_z & 0 \\ 0 & W_D} = \sbm{ M_z & 0 \\ 0 & W_D} \tA$ and $ \tB \sbm{ M_z & 0 \\ 0 & W_D} = \sbm{ M_z & 0 \\ 0 & W_D} \tB$;
 
 \item{(ii)} $\tB = \tA^* \sbm{ M_z & 0 \\ 0 & W_D },$ and $\tB = \tA^* \sbm{ M_z & 0 \\ 0 & W_D }$; and
 
 \item{(iii)} $\| \tA \| \le 1$.
 \end{itemize}
 As in the proof of Theorem \ref{T:WoldPC}, conditions (i) ad (ii) force
 $\tA$ and $\tB$ to have the block-diagonal form 
 $$(\tA,\tB) = \left(\begin{bmatrix}
 M_{\tG_1^*+z\tG_2} & 0 \\ 0 & \tA_{22}
\end{bmatrix}  , \begin{bmatrix} M_{\tG_2^*+z\tG_1}& 0 \\ 0 & \tB_{22}\end{bmatrix}\right).$$
for some pseudo-commutative tetrablock unitary $(\tA_{22},\tB_{22},W_D)$, and operators $\tG_1,\tG_2\in\cB(\cD_{T^*})$ so that the linear pencils $\tG_1^*+z\tG_2$ and $\tG_1^*+z\tG_2$ are contraction-valued for all $z\in\bD$. We now use the fact that the triple \eqref{pcLift} is a pseudo-commutative lift of $(A,B,T)$, i.e., the operators $\tA,\tB$ satisfy the conditions
$$
\begin{bmatrix}
 M_{\tG_1^*+z\tG_2}^* & 0 \\ 0 & \tA_{22}^*
\end{bmatrix} \begin{bmatrix}
\cO_{D_{T^*},T^*}\\ Q_{T^*}
\end{bmatrix}=\begin{bmatrix}
\cO_{D_{T^*},T^*}\\ Q_{T^*}
\end{bmatrix}A^*
$$
and
$$
\begin{bmatrix} M_{\tG_2^*+z\tG_1}^*& 0 \\ 0 & \tB_{22}^*\end{bmatrix}
\begin{bmatrix}
\cO_{D_{T^*},T^*}\\ Q_{T^*}
\end{bmatrix}=\begin{bmatrix}
\cO_{D_{T^*},T^*}\\ Q_{T^*}
\end{bmatrix}B^*.
$$Equivalently,
\begin{align}\label{uniquness1}
&M_{\tG_1^*+z\tG_2}^* \cO_{D_{T^*},T^*}=\cO_{D_{T^*},T^*}A^*,\; M_{\tG_1^*+z\tG_2}^* \cO_{D_{T^*},T^*}=\cO_{D_{T^*},T^*}A^*\\
&\mbox{and}\quad(\tA_{22}^*,\tB_{22}^*,W_D^*)Q_{T^*}=Q_{T^*}(A^*,B^*,T^*).\label{uniquness2}
\end{align}
We first show that $(\tA_{22},\tB_{22})=(R_D,S_D)$. Note that since $\tA_{22}$ commutes with $W_D$ and $W_D$ is a unitary, $\tA_{22}$ commutes with $W_D^*$ 
as well, and so we use \eqref{uniquness2} to compute
\begin{align*}
\tA_{22}^*(W_D^nQ_{T^*}h)=W_D^n\tA_{22}^*Q_{T^*}h&=W_D^nQ_{T^*}A^*h\\
&=W_D^nR_D^*Q_{T^*}h=R_D^*(W_D^nQ_{T^*}h).
\end{align*}Since $\{W_D^nQ_{T^*}h:n\geq0\mbox{ and }h\in\cH\}$ is dense in $\cQ_{T^*}$, we have $\tA_{22}=R_D$. Similarly, $\tB_{22}=S_D$. Next we show that 
$(\tG_1,\tG_2)$ are the fundamental operators of $(A^*,B^*,T^*)$. To this end, we can use \eqref{uniquness1} and the power series expansion of 
$\cO_{D^*,T^*}(z)=\sum_{n\geq0}D_{T^*}T^{*n}h$ to arrive at the equations
$$
\tG_1D_{T^*}+\tG_2^*D_{T^*}T^*=D_{T^*}A^*\quad\text{and}\quad
\tG_2D_{T^*}+\tG_1^*D_{T^*}T^*=D_{T^*}B^*.
$$By part (ii) of Theorem \ref{T:FundTetra} applied to the tetrablock contraction $(A^*,B^*,T^*)$, $(\tG_1,\tG_2)$ must be equal to the Fundamental Operator
pair for $(A^*, B^*,T^*)$.
\end{proof}

\subsection{A Sz.-Nagy--Foias type functional model}
Sz.-Nagy and Foias gave a function-space realization of $\cQ_{T^*}$ and thereby produced a concrete functional model for a contraction $T$. 
In their analysis a crucial role is played by what they called the {\bf characteristic function} associated with $T$:
\begin{align}\label{CharcFunc}
    \Theta_T(z): = -T+z\cO_{D_{T^*}, T^*}T|_{\cD_T}:\cD_T\mapsto \cD_{T^*}.
\end{align}
The name suggests the fact  the characteristic function $\Theta_T$ enables one to write down an explicit functional model on which there is a 
compressed multiplication operator $\bfT$ which recovers the original operator $T$ up to unitary equivalence in case $T$ is a 
c.n.u.\ contraction (see Chapter VI of  \cite{Nagy-Foias}). Let $\Theta_T(\zeta)$ be the radial limit of the characteristic function defined almost everywhere on $\bT$. 
Consider 
\begin{equation}\label{DefectCharc}
\Delta_{T}(\zeta) := (I - \Theta_T(\zeta)^* \Theta_T(\zeta))^{1/2}.
\end{equation}
Sz.-Nagy and Foias showed in \cite{Nagy-Foias} that
$$
V_{\rm NF}:=\begin{bmatrix}
M_z&0\\0&M_{\zeta}|_{\overline{\Delta_T L^2(\cD_T)}} 
\end{bmatrix}:\begin{bmatrix}
H^2(\cD_{T^*})\\ \overline{\Delta_T L^2(\cD_T)}
\end{bmatrix}\to \begin{bmatrix}
H^2(\cD_{T^*})\\ \overline{\Delta_T L^2(\cD_T)}
\end{bmatrix}
$$
is a minimal isometric lift of $T$ via some isometric embedding 
$$
\Pi_{\rm NF}:\cH\to \begin{bmatrix}
H^2(\cD_{T^*})\\ \overline{\Delta_T L^2(\cD_T)}\end{bmatrix}=:\cK_{\rm NF}
$$ such that
\begin{align}\label{RanPinf}
\cH_{\rm NF}:=\operatorname{Ran}\Pi_{\rm NF}=\begin{bmatrix} H^2(\cD_{T^*}) \\ \overline{ \Delta_{T}L^2(\cD_T)} \end{bmatrix}
  \ominus \begin{bmatrix} \Theta_T \\ \Delta_{T} \end{bmatrix} \cdot H^2(\cD_T).
\end{align}
Any two minimal isometric lifts of a given contraction $T$ are unitarily equivalent; see Chapter I of \cite{Nagy-Foias}. In \cite{BS-Memoir} an explicit unitary $u_{\text{min}}:\cQ_{T^*}\to\overline{\Delta_T L^2(\cD_T)}$ is found that intertwines $W_D$ and $M_{\zeta}|_{\overline{\Delta_T L^2(\cD_T)}}$ and
\begin{equation}\label{Pinf}
\Pi_{\rm NF}=\begin{bmatrix} I_{H^2(\cD_{T^*})} & 0 \\ 0 &  u_{\text{min}} \end{bmatrix} \Pi_D.
\end{equation}

It is possible to give a concrete Sz.-Nagy--Foias type functional model using the transition map $u_{\rm{min}}:\cQ_{T^*}\to \overline{\Delta_TL^2(\cD_T)}$ as appeared (see \eqref{Pinf}) in the case of a single contractive operator above. We must replace the canonical tetrablock unitary $(R_D,S_D,W_D)$ by its avatar on the function space $\overline{\Delta_TL^2(\cD_T)}$:
\begin{align}\label{CanonTetraUniNF}
(R_{\rm{NF}}, S_{\rm{NF}},W_{\rm{NF}})=u_{\rm{min}}^*(R_D,S_D,W_D)u_{\rm{min}}.
\end{align}Then the following functional model is a straightforward consequence of Theorem \ref{T:Dmodel} and \eqref{Pinf}.

\begin{thm}\label{T:NFmodel}
Let $(A,B,T)$ be a tetrablock contraction on a Hilbert space $\mathcal{H}$ such that $T$ is c.n.u., and $(G_1,G_2)$ be the fundamental operators of 
$(A^*,B^*,T^*)$. Then $(A,B,T)$ is jointly unitarily equivalent to
\begin{align}\label{FunctModel}
    P_{\cH_{\rm{NF}}}\left(\begin{bmatrix}
M_{G_1^*+zG_2}&0\\0&R_{\rm{NF}}
\end{bmatrix},   
 \begin{bmatrix}  M_{G_2^*+zG_1}&0\\0&S_{\rm{NF}}  \end{bmatrix},   
\begin{bmatrix} M_z&0\\0& W_{\rm{NF}} \end{bmatrix}\right)\bigg|_{\cH_{\rm{NF}}},
\end{align}
where $\cH_{\rm{NF}}$ is the functional model space given by (see \eqref{RanPinf}) 
$$
\cH_{\rm{NF}}=\operatorname{Ran}\Pi_{\rm{NF}}=\begin{bmatrix} H^2(\cD_{T^*}) \\ \overline{ \Delta_{T}L^2(\cD_T)} \end{bmatrix}
  \ominus \begin{bmatrix} \Theta_T \\ \Delta_{T} \end{bmatrix} \cdot H^2(\cD_T).
  $$
\end{thm}

Note that in the special case when $T^{*n}\to 0$ strongly as $n\to\infty$, the space $\cQ_{T^*}=0$ and hence also $\overline{ \Delta_{T}L^2(\cD_T)}  = 0$,
i.e., $\Theta_T$ is inner.  Therefore in this special case the models above simply boil down to the following which was obtained in \cite{SauNYJM}. 

\begin{thm}{\rm (See \cite[Theorem 4.2]{SauNYJM})}\label{T:fm}
 Let $(A,B,T)$ be a pure tetrablock contraction on a Hilbert space $\mathcal{H}$ and $(G_1,G_2)$ be the fundamental operators of $(A^*,B^*,T^*)$. Then $(A,B,T)$ is jointly unitarily equivalent to
$$P_{\operatorname{Ran}\cO_{D_{T^*}, T^*}}(M_{G_1^*+zG_2}, M_{G_2^*+zG_1}, M_z))|_{\operatorname{Ran}\cO_{D_{T^*}, T^*}}.$$
\end{thm}

\section{Tetrablock data sets: characteristic and special}   \label{S:DataSets}
In this section we provide some preliminary results towards a Sz.-Nagy-Foias-type model theory for tetrablock-contraction operator tuples $(A,B,T)$. Note that if $T$ is unitary, then the characteristic function $\Theta_T$, as in \eqref{CharcFunc}, is trivial (i.e., equal to the zero operator between the zero spaces).  As the most general contraction operator $T$
is the direct sum of a unitary $T_u$ with a completely nonunitary (c.n.u.) part $T_{\rm cnu}$ and the model theory for unitary operators is easily handled by spectral theory,
it is natural for model theory purposes to restrict to the case where $T$ is c.n.u.
One then associates a functional model spaces 
$$
 \bcK_T =  \begin{bmatrix}  H^2(\cD_{T^*}   \\  \overline{\Delta_T L^2(\cD_T) } \end{bmatrix} 
\quad  \bcH_T  =\bcK_T \ominus \begin{bmatrix}  \Theta_T  \\  \Delta_T  \end{bmatrix} H^2(\cD_T)  \subset \bcK_T 
$$
together with functional-model operators $\bfT_T$   and $\bfV_T$  by
$$
\bfT_T = P_{\bcH_T} \begin{bmatrix} M_z & 0 \\ 0 & M_\zeta \end{bmatrix} \bigg|_{\bcH_T} \text{ on }  \bcH_T, \quad
\bfV_T =  \begin{bmatrix} M_z & 0 \\ 0 & M_\zeta \end{bmatrix} \text{  on } \bcK_\bT.
$$
Then it is immediate that:
\begin{enumerate}
\item [(NF1)]  $\bfV_T$ is an isometry on $\bcK_T$.
\item[(NF2)] $\bcH_\bfT^\perp = \sbm{ \Theta_T \\ \Delta_T} H^2(\cD_T)$ is invariant for $\bfV_T$, and hence $\bfV_T$ is an isometric lift for $\bfT_T$.
\end{enumerate}
Less immediately obvious are other features of the model:
\begin{enumerate}
\item[(NF3)]  (See \cite[Theorem VI.2.3]{Nagy-Foias}) If $T$ is c.n.u., then  $\bfT_T$ is unitarily equivalent to $T$.

\item[(NF4)] (See \cite[Theorem VI.3.4]{Nagy-Foias}.)
If $T$ on $\cH$ and $T'$ on $\cH'$ are two c.n.u.~contraction operators, then $T$ is unitarily equivalent to $T'$
if and only if $\Theta_T$ coincides with $\Theta_{T'}$ in the following sense: {\em  there exist unitary change-of-coordinate maps
$\phi \colon \cD_T \to  \cD_{T'}$ and $\phi_* \colon \cD_{T^*} \to \cD_{T^{\prime *}}$ so that
$\phi_* \Theta_T(\lam)  = \Theta_{T'}(\lam) \phi$ for all $\lam \in {\mathbb D}$.} Often in the literature this property is described simply as:
{\em the characteristic function $\Theta_T$ is a complete unitary invariant for c.n.u.~contraction operator $T$.}
\end{enumerate}

The Sz.-Nagy-Foias theory goes still further by identifying the {\em coincidence-envelop} of the characteristic functions  \eqref{ThetaT},
i.e., the set of all contractive operator functions $(\cD, \cD_*, \Theta)$ coinciding with the characteristic function $\Theta_T$ for some c.n.u.~contraction operator
$T$,
as simply any contractive operator function $(\cD, \cD_*, \Theta)$ which is \textbf{pure} in the sense that
$$
  \| \Theta(0) u \| < \| u \| \text{ for any } u \in \cD \text{ such that } u \ne 0.
$$
Then we can start with any pure COF $(\cD, \cD_*, \Theta)$, form $\bcK(\Theta)$ and $\bcH(\Theta)$ according to
$$
\bcK(\Theta) = \begin{bmatrix} H^2(\cD_*) \\ \overline{ D_\Theta \cdot L^2(\cD)} \end{bmatrix}, \quad
\bcH(\Theta) = \bcK(\Theta) \ominus \begin{bmatrix}  \Theta \\ D_\Theta \end{bmatrix} H^2(\cD) \subset  \bcK(\Theta)
$$
where $D_\Theta$ is the $\Theta$-defect operator function $D_\Theta(\zeta) = ( I_\cD - \Theta(\zeta)^* \Theta(\zeta))^{\frac{1}{2}}$.
Then we can form the model operators
$$
\bfV(\Theta) = \begin{bmatrix} M_z & 0 \\ 0 & M_\zeta \end{bmatrix} \text{ on } \bcK(\Theta), \quad
\bfT(\Theta) = P_{\bcH(\Theta)} \bfV(\Theta) \bigg|_{\bcH(\Theta)}.
$$
Then we have the additional results:
\begin{enumerate}
\item[(NF5)]  (See \cite[Theorem VI.3.1]{Nagy-Foias}.) Given any pure COF $\Theta$,  $\bfT(\Theta)$ is a c.n.u.~contraction operator on $\bcH(\Theta)$ with characteristic operator function
$\Theta_{\bfT(\Theta)}$ coinciding with $\Theta$.
\end{enumerate}
In this way the loop is closed:  the study of c.n.u.~contraction operators is the same as the study of pure COFs.

To explain generalizations to the setting of tetrablock contractions, we first introduce some useful terminology.
For this discussion, just as in the classical Sz.-Nagy-Foias settings, it makes sense to restrict to c.n.u.~tetrablock contractions
(see Theorem \ref{T:tetra-can-decom} and Remark \ref{R:tetra-can-decom}).

\begin{definition}  \label{D:Edata} Let us say that any collection of objects 
$$
\Xi = (\Theta, (G_1, G_2), \psi \}
$$ 
consisting of
\begin{enumerate}
\item[(i)] a pure COF function $(\cD, \cD_*, \Theta)$,
\item[(ii)] a pair of operators $(G_1, G_2)$ on the coefficient space $\cD_*$, and
\item[(iii)] a measurable function $\psi$ on ${\mathbb T}$ such that, for a.e.~$\zeta \in {\mathbb T}$,
$\psi(\zeta)$ is a contractive normal operator on $\cD_\Theta(\zeta) = \overline{\operatorname{Ran} \, D_\Theta(\zeta)}$.
\end{enumerate}
is a \textbf{tetrablock data set}.
\end{definition}

\begin{remark} \label{R:Edata}
Canonically associated with any such $\psi$ as in item (iii) in Definition \ref{D:Edata} is the tetrablock unitary triple $(R,S,W)$ on the direct integral space 
$\oplus \int_{\mathbb T} \cD_{\Theta(\zeta)} \frac{ |{\tt d} \zeta|}{2 \pi}$ given by
$$ 
R = M_{\psi^* \cdot \zeta}, \quad S = M_\psi,  \quad W = M_\zeta
$$
and (as one sweeps over all possible such $\psi$), this is the general tetrablock unitary operator triple $(R,S,W)$ on the space 
$\oplus \int_{\mathbb T} \cD_{\Theta(\zeta)} \frac{ |{\tt d} \zeta|}{2 \pi}$  with the unitary operator $W$ equal to $W = M_\zeta$
(multiplication by the coordinate function) (see Example \ref{E:model} (2)).  Thus item (iii) in the definition of tetrablock data set can be equivalently rephrased as:

\begin{enumerate}
\item[(iii')]  a tetrablock unitary operator-triple $(R,S,W)$ on the direct-integral space $\int_{\mathbb T} \cD_{\Theta(\zeta)} \frac{|{\tt d} \zeta|}{2 \pi}$ such 
that the last unitary component $W$ is equal to multiplication by the coordinate function $W = M_\zeta$.
\end{enumerate}
However, for convenience of notation, we shall continue to use the notation $(R,S,W)$ for the third component of a tetrablock data set $\Xi = (\Theta, (G_1, G_2),
(R,S,W))$ with the convention (iii') also part of the definition.
\end{remark}

\begin{definition}  \label{D:charEdata}
Given a c.n.u.~tetrablock contraction $(A,B,T)$ we say that $\Xi_{A,B,T} = (\Theta, (G_1, G_2), \psi)$ is the
\textbf{characteristic tetrablock data set} for $(A,B,T)$ if
\begin{enumerate}
\item[(i)] $(\cD, \cD_*, \Theta)$ is equal to the Sz.-Nagy-Foias characteristic function $(\cD_T, \cD_{T^*}, \Theta_T)$ for the c.n.u.~contraction operator $T$,
\item[(ii)] $(G_1, G_2)$ is equal to the Fundamental Operator pair for the adjoint tetrablock contraction $(A^*, B^*, T^*)$, and
\item[(iii)] $(R,S,D)$ is given by
$$
((R,S,W))= (R_{\rm NF}, S_{\rm NF}, W_{\rm NF}):=  u_{\rm min}^* (R_D, S_D, W_D)  u_{\rm min}
$$ 
where $(R_D, S_D, W_D)$ is the tetrablock unitary on $\cQ_{T^*}$ determined by tetrablock contraction
$(A,B,T)$ according to Definition \ref{D:canonical}, and where $u_{\rm min} \colon \overline{\Delta_T L^2(\cD_T)} \to \cQ_{T^*}$
is the unitary identification map identifying the Sz.-Nagy-Foias lifting  residual space  $\overline{\Delta_T L^2(\cD_T)}$ with the
Douglas lifting residual space $\cQ_{T^*}$.
\end{enumerate}
\end{definition}

Then it is clear that the characteristic tetrablock data set $\Xi_{A,B,T}$ for a c.n.u.~tetrablock contraction $(A,B,T)$ is a tetrablock data set. 
The natural notion of equivalence for tetrablock-data sets is the following.

\begin{definition}  \label{D:coincide}
Let $(\mathcal{D},\mathcal{D}_*,\Theta)$, $(\mathcal{D'},\mathcal{D'_*},\Theta')$ be two purely contractive
analytic functions. Let $G_1,G_2\in\cB(\cD_*)$, $G_1',G_2'\in\cB(\mathcal{D'_*})$, and $(R,S,W)$ on $\overline{\Delta_\Theta L^2(\mathcal{D})}$ and $(R',S',W')$ on 
$\overline{\Delta_{\Theta'} L^2(\mathcal{D'})}$ be two tetrablock unitaries (with $W$ and $W'$ equal to  $M_\zeta$ on their respective spaces).
 We say that the two triples $(\Theta, (G_1,G_2),(R,S,W))$
and $(\Theta', (G_1',G_2'),(R',S',W'))$ {\bf coincide} if:
\begin{enumerate}
  \item[(i)] $(\mathcal{D},\mathcal{D}_*,\Theta)$ and $(\mathcal{D'},\mathcal{D'_*},\Theta')$ coincide,
\item[(ii)] the unitary operators $\phi$, $\phi_*$ involved in the coincidence of $(\mathcal{D},\mathcal{D}_*,\Theta)$ and $(\mathcal{D'},\mathcal{D'_*},\Theta')$ satisfy 
the additional intertwining conditions:
\begin{align*}
\phi_*(G_1,G_2)=(G_1',G_2')\phi_*\quad\mbox{and}\quad \omega_{\phi}(R,S,W)=(R',S',W')\omega_{\phi},
\end{align*}
where $\omega_{\phi}:\overline{\Delta_{\Theta} L^2(\mathcal{D})}\to\overline{\Delta_{\Theta'} L^2(\mathcal{D'})}$
is the unitary map induced by $\phi$ according to the formula
\begin{equation}      \label{omega-u}
\omega_{\phi}:=(I_{L^2}\otimes \phi)|_{\overline{\Delta_{\Theta} L^2(\mathcal{D})}}.
\end{equation}
\end{enumerate}
\end{definition}

Given a characteristic tetrablock data set $\Xi_{(A,B,T)} = (\Theta, (G_1, G_2), (R,S,W))$ for a tetrablock contraction $(A,B,T)$ we can write down a functional model:
$$
\bcK(\Xi) = \begin{bmatrix} H^2(\cD_*) \\  \overline{D_\Theta L^2(\cD)} \end{bmatrix}, \quad
\bcH(\Xi) = \bcK \ominus \begin{bmatrix} \Theta \\ D_\Theta \end{bmatrix} H^2(\cD) 
$$
with functional-model operators
\begin{align*}
& \bfV(\Xi) = \bigg( \begin{bmatrix} M_{G_1^* + z G_2} & 0 \\ 0 & R \end{bmatrix}, \begin{bmatrix} M_{G_2^* + z G_1} & 0 \\ 0 & S \end{bmatrix},
\begin{bmatrix} M_z & 0 \\ 0 & M_\zeta \end{bmatrix} \bigg) \text{ on } \bcK(\Xi),   \\
& \bfT(\Xi) = P_{\bcH(\Xi)} \bfV(\Xi) |_{\bcH(\Xi)}.
\end{align*}
The tetrablock analogue of items (NF1)-(NF2) in our discussion of the Sz.-Nagy-Foias model above  fails without the additional assumptions: 
 namely, it is not the case that $\bfV$ is a tetrablock isometry as well as that $\bfV$ is a lift for $\bfT$
unless we also impose the condition \eqref{com=} on $(G_1, G_2)$.  Nevertheless, the analogue of (NF3) does hold:  {\em given that $\Xi$ is the
characteristic tetrablock triple for $(A,B,T)$, it is the case that $(A,B,T)$ is unitarily equivalent to $\bfT(\Xi)$}:  this is the content of the last part of
item (3) in Theorem \ref{T:Dmodel} (after a conversion to Sz.-Nagy-Foias rather than Douglas coordinates): see \eqref{SemiModel} and \eqref{HD}. 
The next theorem  amounts to the analogue of (NF4) in our list of features for the Sz.-Nagy-Foias model above.

\begin{thm}\label{Thm:CompUniInv}
Let $(A,B,T)$ and $(A',B',T')$ be two tetrablock contractions acting on $\cH$ and $\cH'$, respectively. Let 
$$
(\Theta_T, (G_1,G_2), (R_{\rm NF}, S_{\rm NF},W_{\rm NF}), \quad  (\Theta_{T'}, (G_1',G_2'), (R_{\rm NF}', S_{\rm NF}',W_{\rm NF}'))
$$ 
be the characteristic tetrablock data sets for  $(A,B,T)$ and $(A',B',T')$, respectively.

\smallskip

\noindent
{\rm (1)} If $(A,B,T)$ and $(A',B',T')$ are unitarily equivalent, then their characteristic tetrablock data sets coincide.

\smallskip

\noindent
{\rm (2)}
Conversely, if $T$ and $T'$ are c.n.u.\ contractions and the characteristic tetrablock data sets of $(A,B,T)$ and $(A',B',T')$ coincide, then $(A,B,T)$ and $(A',B',T')$ 
are unitarily equivalent.
\end{thm}

\begin{proof}
First suppose that $(A,B,T)$ and $(A',B',T')$ are unitarily equivalent via a unitary $\tau:\cH\to\cH'$. The fact that $\Theta_T$ and $\Theta_{T'}$ coincide is a part of the Sz.-Nagy--Foias theory \cite{Nagy-Foias}. Indeed, note that
\begin{align*}
\tau(I-T^*T)=(I-T'^*T')\tau \text{ and }\tau(I-TT^*)=(I-T'T'^*)\tau
\end{align*}
and therefore by the functional calculus for positive operators,
\begin{align}\label{IntwinDefects}
    \tau D_T=D_{T'}\tau\quad\mbox{and}\quad\tau D_{T^*}=D_{T'^*}\tau
\end{align} and thereby inducing two unitary operators
\begin{align}\label{u&u*}
\phi:=\tau|_{\cD_T}:\cD_T\to\cD_{T'} \text{ and }\phi_*:=\tau|_{\cD_{T^*}}:\cD_{T^*}\to\cD_{T'^*}.
\end{align}
Consequently $\phi_*\Theta_{T}=\Theta_{T'}\phi$. Next, since the fundamental operators are the unique operators that satisfy the equations for $(X_1,X_2)=(G_1,G_2)$:
\begin{align*}
    A^*-BT^*=D_{T^*}X_1D_{T^*} \quad\mbox{and}\quad B^*-AT^*=D_{T^*}X_2D_{T^*},
\end{align*} 
one can easily obtain using (\ref{IntwinDefects}) that
\begin{align}\label{IntwinFunds}
\phi_*(G_1,G_2)=(G_1',G_2') \phi_*.
\end{align} 
Finally the proof of the forward direction will be complete if we establish that 
\begin{align}\label{ForwardLast}
(R_{\rm NF},S_{\rm NF},W_{\rm NF})=\omega_\phi^*(R'_{\rm NF},S'_{\rm NF},W'_{\rm NF})\omega_\phi
\end{align} 
where $\omega_\phi=(I_{L^2}\otimes \phi)|_{\overline{\Delta_{T} L^2(\mathcal{D}_T)}}:
\overline{\Delta_{T} L^2(\mathcal{D}_T)}\to \overline{\Delta_{T'} L^2(\mathcal{D}_{T'})}$. 
For this we first note that
\begin{align*}
\begin{bmatrix}
I&0\\0& u_{\rm min}
\end{bmatrix}\begin{bmatrix}
\cO_{D_{T^*},T^*}\\ Q_{T^*}
\end{bmatrix}=\Pi_{\rm NF}&=\begin{bmatrix}
I_{H^2}\otimes \phi_*^*&0\\0& \omega_\phi^*
\end{bmatrix}\Pi_{\rm NF}'\tau \\
&=\begin{bmatrix}
I_{H^2}\otimes \phi_*^*&0\\0& \omega_\phi^*
\end{bmatrix}\begin{bmatrix}
I&0\\0& u_{\rm min}'
\end{bmatrix}\begin{bmatrix}
\cO_{D_{T'^*},T'^*}\\ Q_{T'^*}
\end{bmatrix}\tau,
\end{align*}
from which we read off that 
\begin{align}\label{u-minQ}
u_{\rm min}Q_{T^*}=\omega_\phi^* u_{\rm min}'Q_{T'^*}\tau.
\end{align}
Now since $\cQ_{T'^*}=\overline{\operatorname{span}}\{W_D'^nQ_{T'^*}\tau h:h\in\cH,n\geq 0\}$ and $u_{\rm min}$ has the intertwining property $u_{\rm min} W_D=M_\zeta u_{\rm min}$, we use \eqref{u-minQ} to compute
\begin{align*}
    \omega_{\phi}\cdot u_{\rm min}\cdot \omega_\tau^* (W_D'^nQ_{T'^*}\tau h)
    &=\omega_{\phi}\cdot u_{\rm min} (W_D^nQ_{T^*}h)\\
    &=\omega_{\phi} M_\zeta ^n u_{\rm min} Q_{T^*}h\\
    &=M_\zeta^n\omega_\phi \cdot u_{\rm min} Q_{T^*}h\\
    &=M_\zeta^n u_{\rm min}' Q_{T'^*}\tau h=u_{\rm min}' (W_D'^nQ_{T'^*}\tau h).
\end{align*}Consequently
\begin{align}\label{u-min}
\omega_\phi\cdot u_{\rm min}\cdot \omega_\tau^*= u_{\rm min}'.
\end{align}Using this identity and the intertwining properties of the unitaries involved, it is now easy to establish \eqref{ForwardLast}.

Conversely, suppose $T$ and $T'$ are c.n.u.\ contractions, $\phi:\cD_T\to\cD_{T'}$ and $\phi_*:\cD_{T^*}\to\cD_{T'^*}$ be the unitary operators involved 
in the coincidence of the characteristic tetrablock data sets 
$$
((G_1,G_2), (R_{\rm NF}, S_{\rm NF},W_{\rm NF}),\Theta_T), \quad 
((G_1',G_2'), (R_{\rm NF}', S_{\rm NF}',W_{\rm NF}'),\Theta_{T'}).
$$
 By Definition \ref{D:coincide}, it follows that the unitary
$$
U=\begin{bmatrix} I_{H^2}\otimes\phi_* & 0 \\ 0& \omega_\phi \end{bmatrix} \colon
\begin{bmatrix} H^2(\cD_{T^*}) \\ \overline{ \Delta_{T}L^2(\cD_T)} \end{bmatrix} \to  
\begin{bmatrix} H^2(\cD_{T'^*}) \\ \overline{ \Delta_{T'}L^2(\cD_{T'})} \end{bmatrix}
$$
identifies the model spaces $\cH_{\rm NF}$ and $\cH_{\rm NF}'$ and intertwines the model operators as in \eqref{FunctModel} associated with 
$(A,B,T)$ and $(A',B',T')$, respectively. This completes the proof of Theorem \ref{Thm:CompUniInv}.
\end{proof}

The question remains
as to what additional coupling conditions must be imposed on a tetrablock data set $\Xi$ to assure that $\Xi$ coincides with the characteristic tetrablock data set
for a c.n.u.~ tetrablock contraction $(A,B,T)$.   In the Sz.-Nagy-Foias theory, the data set (or invariant) consists of a single COF, and the only
additional requirement is that it must be pure.

From the results of Section \ref{SS:FundOps} we see that any characteristic tetrablock data
$$
\Xi_{A,B,T} =  (\Theta, (G_1, G_2), (R,S,W))
$$
 for  a c.n.u.~tetrablock contraction operator-triple $(A,B,T)$ satisfies the additional conditions
 (expressed directly in terms of the components of  $\Xi_{A,B,T}$ rather than in terms of $(A,B,T)$):
\begin{enumerate}
\item[(i)] $\Theta$ is a pure COF (see \cite[Theorem VI.3.1]{Nagy-Foias}),

\item[(ii)]  the numerical radius conditions
$$
\nu(G_1^* + z G_2) \le 1, \quad \nu(G_2^* + z G_1) \le 1 \text{ for all } z \in \overline{\mathbb D}
$$
hold,  implying that also the spectral radius conditions
$$
r(M_{G_1^* + z G_2}) \le 1, \quad  r(M_{G_2^* + z G_1}) \le 1
$$
(see Theorems \ref{T:pseudo-strict} (1)).
\item[(iii)] It is almost the case that the spectral radius and norm agree for $M_{G_1^* + z G_2}$ and $M_{G_2^* + z G_1}$ in the following sense 
(see Theorem \ref{T:pc-vs-strict-Eisom} (1)):
$$
r(M_{G_1^* + z G_2} \cdot M_{G_2^* + z G_1} ) = \max\{ \| M_{G_1^* + z G_2} \|^2,  \, \|M_{G_2^* + z G_1} \|^2 \}.
$$
\end{enumerate}
However we do not expect that just imposing these conditions is sufficient to guarantee that such a tetrablock data set $\Xi$  will coincide with
the characteristic tetrablock data set for some tetrablock contraction, so we do not expect to have an analogue of (NF5) at this level of generality.

Let us now specialize our class of tetrablock contractions to what we call \textbf{special tetrablock contractions}, i.e., any  tetrablock contraction
$(A,B,T)$ with the special property that the Fundamental Operator pair $(G_1, G_2)$ for $(A^*, B^*, T^*)$ satisfies the additional pair of operator equations
\eqref{com=}.  By Theorem \ref{T:Dmodel}, this is equivalent to $(A,B,T)$ having a minimal tetrablock isometric lift $(V_1, V_2, V_3)$ acting on a
minimal Sz.-Nagy isometric-lift space for $T$ with $V_3$ equal to a minimal isometric lift for the single contraction operator $T$.

This suggests that we define a \textbf{special tetrablock data set}  as follows. For convenience in later discussion we shall now write the last component
$(R,S,W)$ simply as $\psi$ for a measurable  contractive-normal operator-valued function $\zeta \mapsto \psi(\zeta) \in \cB(\cD_{\Theta(\zeta)})$
according to the convention explained in Remark \ref{R:Edata}.

\begin{definition}\label{D:Adm}
We say that the tetrablock data set
\begin{equation}   \label{Edata}
 \Xi = ((\cD, \cD_*, \Theta), (G_1, G_2), \psi)
 \end{equation}
 is a \textbf{special tetrablock data set}  if the following conditions hold:
 \begin{enumerate}
\item[(i)] The operators $G_1, G_2 \in \cB(\cD_*)$ satisfy the commutativity conditions \eqref{com=}, i.e.,
$$
[G_1,G_2]=0,\quad [G_1^*,G_1]=[G_2^*,G_2]
$$
as well as the pencil-contractivity condition
$$
\|G_1^*+zG_2\|\leq 1 \text{ for all } z\in\overline{\bD} 
$$
and then also 
$$
 \|G_2^*+zG_1 \|\leq 1 \text{ for all } z \in \overline{{\mathbb D}}.
$$

\item[(ii)] the space $\left\{\sbm{\Theta\\ D_\Theta}f:f\in H^2(\cD)\right\}$ is jointly invariant under the operator triple
\begin{align}\label{AdmLifts}
\bigg( \begin{bmatrix} M_{G_1^* + zG_2} & 0 \\ 0 & M_{\psi(\zeta)^*\cdot \zeta} \end{bmatrix}, \begin{bmatrix} M_{G_2^* + zG_1} & 0 \\ 0 & M_{\psi(\zeta)}
 \end{bmatrix},
\begin{bmatrix} M_z & 0 \\ 0 & M_\zeta \end{bmatrix} \bigg).
\end{align}
\end{enumerate}
\end{definition}

\smallskip

Given a special tetrablock data set  $(\Theta, (G_1,G_2),\psi)$, we say that the space
\begin{align}\label{AdmSpace}
{\boldsymbol \cH} = \begin{bmatrix} H^2 (\cD_*) \\ \overline{ D_\Theta L^2(\cD)} \end{bmatrix} \ominus
\begin{bmatrix} \Theta \\ D_\Theta \end{bmatrix} H^2(\cD)
\end{align}
is the {\bf functional model space} and the (commutative) operator triple  $(\bA, \bB, \bfT)$ given by
\begin{align}\label{AdmOps}
P_{\bcH}\bigg( \begin{bmatrix} M_{G_1^* + zG_2} & 0 \\ 0 & M_{\psi(\zeta)^* \cdot \zeta} \end{bmatrix}, \begin{bmatrix} M_{G_2^* + zG_1} & 0 \\ 0 & M_{\psi(\zeta)} \end{bmatrix},
\begin{bmatrix} M_z & 0 \\ 0 & M_\zeta \end{bmatrix} \bigg) \bigg|_\bcH
\end{align}
the {\bf functional-model operator triple} associated with the data set.  The following theorem is our one analogue of item (NF5) in our list of features
of the Sz.-Nagy-Foias model.

\begin{thm}
If $\Xi = (\Theta, (G_1,G_2),\psi)$ is a special tetrablock data set, then the associated model operator triple $(\bf A,\bf B, \bf T)$ as in \eqref{AdmOps}, 
is a tetrablock contraction that lifts to the tetrablock isometry as in \eqref{AdmLifts}. Moreover, the tetrablock data set $\Xi$ coincides with the characteristic 
triple of $(\bf A, \bf B, \bf T)$.
\end{thm}
\begin{proof}
By part (i) of Definition \ref{D:Adm}, the triple as in \eqref{AdmLifts} is a strict tetrablock isometry, and by part (ii), it lifts the model triple $(\bf A, \bf B, \bf T)$ as in \eqref{AdmOps}. Thus in particular, $(\bf A, \bf B, \bf T)$ is a tetrablock contraction and the first part of the theorem follows. For the second part, we use the Sz.-Nagy--Foias model theory for single contractions and Theorem \ref{T:Dmodel} as follows. Apply Theorem VI.3.1 in \cite{Nagy-Foias} to the purely contractive analytic function $\Theta$ to conclude that the characteristic function $\Theta_{\bf T}$ of $\bf T$ coincides with $\Theta$, i.e., there exists unitary operators $u:\cD\to\cD_{\bf T}$ and $u_*:\cD_*\to\cD_{\bf T^*}$ such that $u_*\cdot\Theta(z)=\Theta_{\bf T}(z)\cdot u$ for all $z\in\bD$. Let us set $(G_1',G_2'):=u_*(G_1,G_2)u_*^*$ and $(R',S',W'):=\omega_u(M_{\psi(\zeta)^*\cdot \zeta},M_{\psi(\zeta)},M_\zeta)\omega_u^*$. Since $G_1,G_2$ satisfy the commutativity conditions, $G_1',G_2'$ also satisfy the same conditions, and consequently the triple
\begin{align}\label{ConjLift}
\bigg( \begin{bmatrix} M_{G_1'^* + zG_2'} & 0 \\ 0 & \omega_u M_{\psi(\zeta)^*\cdot \zeta}\omega_u^* \end{bmatrix}, \begin{bmatrix} M_{G_2'^* + zG_1'} & 0 \\ 0 & \omega_u M_\psi\omega_u^* \end{bmatrix},
\begin{bmatrix} M_z & 0 \\ 0 & M_\zeta \end{bmatrix} \bigg)
\end{align}is a strict tetrablock isometry. Note that  since
\begin{align*}
\begin{bmatrix}
 u_*&0\\0&\omega_u
 \end{bmatrix}\begin{bmatrix}
 \Theta\\ \Delta_\Theta
 \end{bmatrix}=\begin{bmatrix}
 \Theta_{\bf T}\\ \Delta_{\Theta_{\bf T}}
 \end{bmatrix}\begin{bmatrix}
 u&0\\0&\omega_u\end{bmatrix},
\end{align*} the unitary operator 
$$\tau:=\begin{bmatrix}u_*&0\\o&\omega_u\end{bmatrix}:\begin{bmatrix}H^2(\cD_*)\\ \overline{ \Delta_\Theta L^2(\cD)} \end{bmatrix}
\to
 \begin{bmatrix}H^2(\cD_{\bf T^*})\\ \overline{ \Delta_{\Theta_{\bf T}} L^2(\cD_{\bf T})}
\end{bmatrix}$$takes the functional model space $\bcH$ as in \eqref{AdmSpace} onto
\begin{align*}
\begin{bmatrix} H^2 (\cD_{\bf T^*}) \\ \overline{ \Delta_{\Theta_{\bf T}} L^2(\cD_{\bf T})} \end{bmatrix} \ominus
\begin{bmatrix} \Theta_{\bf T} \\ \Delta_{\Theta_{\bf T}} \end{bmatrix} H^2(\cD_{\bf T}).
\end{align*}Therefore by part (ii) of Definition \ref{D:Adm}, the tetrablock isometry as in \eqref{ConjLift} is a lift of $(\bf A, \bf B, \bf T)$ via the embedding $\iota\cdot \tau|_\bcH$, where 
$$\iota:\begin{bmatrix}H^2 (\cD_{\bf T^*}) \\ \overline{ \Delta_{\Theta_{\bf T}} L^2(\cD_{\bf T})} \end{bmatrix} \ominus
\begin{bmatrix} \Theta_{\bf T} \\ \Delta_{\Theta_{\bf T}} \end{bmatrix} \to
\begin{bmatrix}H^2(\cD_{\bf T}) \\ \overline{ \Delta_{\Theta_{\bf T}} L^2(\cD_{\bf T})} \end{bmatrix}$$is the inclusion map. Since
$$
\begin{bmatrix}
M_z&0\\0&M_\zeta
\end{bmatrix}:\begin{bmatrix}H^2(D_{\bf T}) \\ \overline{ \Delta_{\Theta_{\bf T}} L^2(\cD_{\bf T})} \end{bmatrix}\to
\begin{bmatrix}H^2(D_{\bf T}) \\ \overline{ \Delta_{\Theta_{\bf T}} L^2(\cD_{\bf T})} \end{bmatrix}
$$is a minimal isometric lift of $\bf T$, by part (2) of Theorem \ref{T:Dmodel}, there is a unique such tetrablock isometric lift with the last entry of the lift fixed.
 If $\bf G_1,\bf G_2$ are the Fundamental Operators of $(\bf A^*,\bf B^*, \bf T^*)$, then by Theorem \ref{T:NFmodel}, 
\begin{align*}
\left(\begin{bmatrix}
M_{\bf G_1^*+z\bf G_2}&0\\0&R_{\rm{NF}}
\end{bmatrix},   
 \begin{bmatrix}  M_{\bf G_2^*+z\bf G_1}&0\\0&S_{\rm{NF}}  \end{bmatrix},   
\begin{bmatrix} M_z&0\\0& W_{\rm{NF}} \end{bmatrix}\right)
\end{align*}
is another tetrablock isometric lift of $(\bf A,\bf B, \bf T)$, where $(R_{\rm NF}, S_{\rm NF},W_{\rm NF})$ is the canonical tetrablock unitary associated 
with $(\bf A,\bf B, \bf T)$. Therefore we must have
$({\bf G}_1,{\bf G}_2)=(G_1',G_2')=u_*(G_1,G_2)u_*^*$ and  
$$
(R_{\rm NF}, S_{\rm NF},W_{\rm NF})=\omega_u(M_{\psi(\zeta)^*\cdot \zeta},M_{\psi(\zeta)},M_\zeta)\omega_u^*.
$$
This is what was needed to be shown.
\end{proof}

\noindent
\textbf{Epilogue.}  It is easy to write down tetrablock data sets as in \eqref{Edata}.  Given such a tetrablock data set, it may not be very tractable to determine
if in addition it satisfies conditions (i) and (ii) in Definition \ref{D:Adm}.  

However it is not so difficult to cook up viable examples.  For example,
 we note that the commutativity conditions in (i) are automatic if we choose $G_1$ and $G_2$ to be scalar operators on $\cD_*$.
We can arrange the pencil contractivity condition  to hold just by choosing $G_1$ and $G_2$ to be sufficiently small.  If we choose the operator
$\psi(\zeta)$ to be a scalar for each $\zeta$, then we are forced to choose $\psi(\zeta) = G_2^* \zeta + G_1$.  Since we have already chosen $G_2$
and $G_1$ so that the pencil contractivity condition holds, then $\psi(\zeta)$ is contractive and of course a scalar operator is normal.  Then all conditions
are satisfied. In this way we get a whole class of tractable examples of special tetrablock contractions for any pure COF $\Theta$.
If $\cD_*$ and $\cD$ are both at most one-dimensional,  all the examples are of this form.

\smallskip

\noindent
\textbf{Conflict of Interest/Dataset Statement.}  The authors state that there are no conflicts of interest.  No data sets were generated or analyzed during the 
current study.

\end{document}